\documentclass[12pt]{amsart}
\usepackage[margin=1in]{geometry}
\usepackage{amsmath,amsfonts,graphicx}
\usepackage[utf8]{inputenc}
\usepackage{amsthm}
\usepackage{amssymb}
\usepackage{enumerate}
\usepackage{tikz-cd}
\usepackage{float}
\usepackage{hyperref}
\usepackage[capitalise]{cleveref}
\usepackage[utf8]{inputenc}
\usepackage[english]{babel}
\usepackage{lipsum}
\graphicspath{{./images/}}

\title{Dehn Functions of Finitely Presented Metabelian Groups}
\author{Wenhao Wang}
\address{Department of Mathematics\\
  Vanderbilt University\\
 Nashville, TN 37240}
\email[W.~Wang]{wenhao.wang@vanderbilt.edu}

\newtheorem{theorem}{Theorem}[section]
\newtheorem{lemma}[theorem]{Lemma}
\newtheorem{corollary}[theorem]{Corollary}

\theoremstyle{definition}
\newtheorem{proposition}[theorem]{Proposition}
\newtheorem{definition}[theorem]{Definition}

\newtheorem{ques}[theorem]{Question}

\theoremstyle{remark}

\newtheorem*{rems}{Remark}

\newcommand{\llangle}{\langle \langle}
\newcommand{\rrangle}{\rangle \rangle}

\DeclareMathOperator{\area}{Area}

\DeclareMathOperator{\supp}{supp}
\DeclareMathOperator{\tor}{tor}
\DeclareMathOperator{\Tr}{Tr}
\DeclareMathOperator{\sgn}{sgn}

\begin{document}
\maketitle

\begin{abstract}
  In this paper, we compute an upper bound for the Dehn function of a finitely presented metabelian group. In addition, we prove that the same upper bound works for the relative Dehn function of a finitely generated metabelian group. We also show that every wreath product of a free abelian group of finite rank with a finitely generated abelian group can be embedded into a metabelian group with exponential Dehn function.
\end{abstract}

\section{Introduction}
\label{intro}

The class of groups we are interested in is the class of finitely presented metabelian groups. Recall that a group is \emph{metabelian} if its derived subgroup is abelian and a group is \emph{finitely presented} if it is generated by a finite set subject to finitely many defining relations. For finitely presented groups, there is a geometric and combinatorial invariant called the Dehn function. The Dehn function of a finite presented group $G=\langle X\mid R\rangle$ is defined to be 
\[\delta(n)=\sup_{|w|\leqslant n}\inf\{k\mid w=\prod_{i=1}^k r_i^{f_i} \text{ where } r_i\in R\cup R^{-1}, f_i\in G\}.\]

It was introduced by computer scientists Madlener and Otto to describe the complexity of the word problem of a group \cite{MadlenerOtto1985}, and also by Gromov as a geometric invariant of finitely presented groups \cite{gromov1996geometric}. There are a lot of significant results about Dehn functions in the past 30 years, revealing the relationship between this geometric invariant and algebraic properties. For instance, a finite generated group is hyperbolic if and only if it has sub-quadratic Dehn function \cite{gromov1996geometric}, \cite{yu1991hyperbolicity}. It is also well-known that a finitely presented group has a decidable word problem if and only if its Dehn function is bounded above by a recursive function \cite{MadlenerOtto1985}. The word problem for any finitely generated metabelian group is decidable, which follows from the fact that finitely generated metabelian groups are residually finite. Hence the Dehn function of a finitely presented metabelian group is always bounded above by a recursive function.

In this paper, we are most interested in the asymptotic behavior of Dehn functions rather than the explicit Dehn function for a particular presentation of a group. In order to compare different functions asymptotically, we define the following relation: for $f,g:\mathbb N\to \mathbb N$, we write $f\preccurlyeq g$ if  there exists $C>0$ such that for all $n$,
\[f(n)\leqslant Cg(Cn)+Cn+C.\]
And we say $f\approx g$ if $f\preccurlyeq g$ and $g \preccurlyeq f$. One can verify that $\approx$ is an equivalence relation. This relation has many nice properties. One of which is that it distinguishes polynomials of different degrees while all polynomials of the same degree are equivalent. Moreover, for exponential functions, we have $a^n\approx b^n$ for $a,b>1$ and $n^m\prec a^n$ for $m>0,a>1$. Despite the dependence of Dehn function on finite presentations of a group, all Dehn functions of the same finitely presented group are equivalent under $\approx$ \cite{gromov1996geometric}.

In this paper, we show that there exists a universal upper bound for Dehn functions of all finitely presented metabelian groups (up to equivalence). For example, one choice of such upper bound is the doubly exponential function $2^{2^{n}}$. 

First, let us give some examples of Dehn functions of finitely presented metabelian groups. 
\begin{enumerate}[(1)]
	\item The first class of examples is the class of metabelian Baumslag-Solitar groups $BS(1,n), n\geqslant 2$, which has the presentation
\[BS(1,n)=\langle a,t\mid tat^{-1}=a^n\rangle,\]
for any $n\geqslant 2$. It is well-known that metabelian Baumslag-Solitar groups have exponential Dehn function up to equivalence. The proof can be found in many places, for example, \cite{GrovesHermiller2001}, \cite{ArzhantsecaOsin2002}.
	\item Another class of important examples is called Baumslag's groups which were first introduced by Baumslag in 1972 \cite{baumslag1972finitely}, 
 \[\Gamma=\langle a,s,t \mid [a,a^t]=1,[s,t]=1, a^s=aa^t \rangle \text{ and } \Gamma_m=\langle \Gamma\mid a^m=1\rangle.\] 
 Note that $\Gamma$ contains a copy of $\mathbb Z\wr \mathbb Z$ while $\Gamma_m$ contains a copy of the Lamplighter group $\mathbb Z_m\wr \mathbb Z$. Recall the \emph{wreath product} $A\wr T$ is defined to be the semi-direct product of $\oplus_{t\in T}A^t$ by $T$ with the conjugation action. $\Gamma$ is the first example of a finitely presented group with an abelian normal subgroup of infinite rank. M. Kassabov and T. R. Riley \cite{KassabovRiley2010} showed that $\Gamma$ has an exponential Dehn function while the Dehn function of $\Gamma_m$ is at most $n^4$. In particular, Y. de Cornulier and R. Tessera \cite{CornulierYvesTessera2010} showed that $\Gamma_p$ has a quadratic Dehn function when $p$ is a prime number.
  \item The third example consists of groups that are a semidirect product of a finitely generated free abelian group and cyclic group, namely, $\mathbb Z^n\rtimes \mathbb Z$. Bridson and Gersten have shown that the Dehn function of such groups are either polynomial or exponential depending on the action of $\mathbb Z$ on $\mathbb Z^n$ \cite{bridson1996optimal}.
 \item Lattices in $\mathbb R^n\rtimes_\alpha \mathbb R^{n-1},n\geqslant 3$, have quadratic Dehn function \cite{gromov1996geometric}, where $\alpha:\mathbb R^{n-1}\to GL(n,\mathbb R)$ is an injective homomorphism whose image consists of all diagonal matrices with diagonal entries $(e^{t_1},e^{t_2},\dots,e^{t_n})$ verifying $t_1+t_2+\dots+t_n=0$. Drutu extends the result for the case that $a_1t_1+a_2t_2+\dots+a_nt_n=0$ for any fixed vector $(a_1,a_2,\dots,a_n)$ with at least three nonzero components \cite{dructu2004filling}.  
 \item Let 
	 \[G=\langle a,b,t\mid [a,b]=1,a^t=ab,b^t=ab^2\rangle.\]
	 $G$ is metabelian and polycyclic and it is also the fundamental group of a closed, orientable fibred 3-manifold. It has been shown that $G$ has exponential Dehn function \cite{baumslag1993isoperimetric}. The lower bound can be proved using the structure of the second homology of $G$. 
\end{enumerate}

The Dehn functions of those examples are not very large. In fact, all known cases of finitely presented metabelian groups have at most exponential Dehn functions. It is natural to ask 

\begin{ques}
	Is the Dehn function of any finitely presented metabelian group bounded above by the exponential function?
\end{ques}

The upper bound we obtained in this paper is slightly bigger than the exponential functions. It remains unknown if there exists a metabelian group with a Dehn function that exceeds $2^n$. We will talk about all the obstructions to seeking such groups in the last section. In this paper, we in fact show that it costs at most exponential many relations, with respect to the length of the word, commutting two commutators of a finitely presented metabelian group (See in \cref{metabelianness}). From those pieces of evidence, we find that the only hope for seeking such a group is to find a complex enough membership problem for a submodule in a free module over a group ring (See in \cref{examples2}). 

For other varieties of solvable groups, the question has been studied extensively. It is not hard to show that Dehn functions of finitely generated abelian groups are asymptotically bounded above by $n^2$. For varieties of solvable groups of derived length three or higher, we have

\begin{theorem}[O. Kharlampovich, A. Myasnikov, M. Sapir\cite{KharlampovichMiasnikovSapir2012}]
For every recursive function $f$, there is a residually finite finitely presented solvable group $G$ of derived length 3 with Dehn function greater than $f$.
\end{theorem}

Our main result is the following:

\begin{theorem}
	\label{main}
	Let $G$ be a finitely presented metabelian group. Let $k$ be the minimal torsion-free rank of an abelian group $T$ such that there exists an abelian normal subgroup $A$ in $G$ satisfying $G/A\cong T$.
	
	Then the Dehn function of $G$ is asymptotically bounded above by 
	\begin{enumerate}
		\item $n^2$ if $k=0$;
		\item $2^{n^{2k}}$ if $k>0$.
	\end{enumerate}
\end{theorem}

In particular, if we let $T=G_{ab}$, the abelianizition of $G$, and we suppose $T$ is infinite, $\delta_G$ is asymptotically bounded by $2^{n^{2k}}$ where $k$ is the torsion-free rank of the abelianizition $G_{ab}$. The proof of \cref{main} can be found in \cref{maintheorem1}.

This theorem immediately gives us a uniform upper bound for Dehn functions of finitely presented metabelian groups. For example, let $H(n):\mathbb N\to \mathbb N$ be any super-polynomial function, then $2^{H(n)}$ will be an upper bound of the Dehn function of any finitely presented metabelian group. 

Since metabelian groups form a variety, we also discuss the Dehn function relative to the variety of metabelian groups (defined in \cref{relativedehn2}). Since finitely generated metabelian groups satisfy the maximal condition for normal subgroups \cite{Hall1954}, a finitely generated metabelian group is always finitely presentable in the variety of metabelian groups. Therefore, the relative Dehn function exists for any finitely generated metabelian group. The relative Dehn function $\tilde \delta_G(n)$ has a close connection to the complexity of the membership problem of the submodule $G'$ over the group ring of $G/G'$. And we can translate this connection to a connection between relative Dehn functions and Dehn functions. We prove that the Dehn function $\delta_G(n)$ is bounded above by $\max\{\tilde \delta_G^3(n^3),2^n\}$ (\cref{main3}). We also improve a result in \cite{Fuh2000}, as in \cref{improved}, which gives a better estimate of the relative Dehn function in a special case. In addition, we prove that the same upper bound works for the relative case, i.e.,

\begin{theorem}
\label{relativeDehnfuntion}
	Let $G$ be a finitely generated metabelian group. Let $k$ be the minimal torsion-free rank of an abelian group $T$ such that there exists an abelian normal subgroup $A$ in $G$ satisfying $G/A\cong T$.
	
	Then the relative Dehn function of $G$ is asymptotically bounded above by 
	\begin{enumerate}
		\item $n^2$ if $k=0$;
		\item $2^{n^{2k}}$ if $k>0$.
	\end{enumerate}
\end{theorem}

The proof of \cref{relativeDehnfuntion} can be found in \cref{relativedehn3}.

The general method we establish in this paper provides a way to compute the Dehn function of metabelian groups. In the last section, we generalize one result in \cite{KassabovRiley2010}, and show the following.

\begin{theorem}
\label{embeddingWreathproduct}
Every wreath product of a free abelian group of finite rank with a finitely generated abelian group can be embedded into a metabelian group with exponential Dehn function. In particular, any free metabelian group of finite rank is a subgroup of a metabelian group with exponential Dehn function.
\end{theorem}


\emph{The structure of this paper.} We will state some preliminary concepts and lemmas in \cref{prelim}. In \cref{membership}, we prove \cref{division} that allows us to solve the membership problem for some special free modules in a reasonable time. Then we shall briefly revisit the proof of ``if" part of a result of Bieri and Strebel \cite{bieri1980valuations} in \cref{finitepresented}, which characterizes finitely presented metabelian groups. In \cref{maintheorem}, we provide the proof of \cref{main}. We will talk about relative Dehn functions and the connections between them and Dehn functions in \cref{relativedehn}. In \cref{examples}, we prove \cref{embeddingWreathproduct} and discuss the obstructions for finding metabelian group with the Dehn function greater than the exponential function.

\emph{Acknowledgement.} I would like to thank my advisor Mark Sapir who encourage me to study this question and instructs me. I also would like to thank Nikolay Romanovskiy who kindly answered my question about an algorithm solving the membership problem in modules.  

\section{Preliminaries}
\label{prelim}
\subsection{Notation}
\label{prelim1}

We denote the set of rational integers by $\mathbb Z$ and the set of real numbers by $\mathbb R$. $\mathbb N$ indicates the set of natural numbers, where our convention is that $0\notin \mathbb N$. In addition, we let $\mathbb R^+=\{x\in \mathbb R\mid x>0\}.$ 

If $G$ is a group, we will denote $G'=[G, G]$ to be the derived (commutator) subgroup, $G_{ab}\cong G/G'$ to be the abelianization. For elements $x,y\in G, n\in \mathbb N$, our conventions are $x^{ny}=y^{-1}x^ny, [x,y]=x^{-1}y^{-1}xy$. We use double bracket $\llangle \cdot \rrangle_G$ to denote the normal closure of a set in the group $G$. Sometimes we omit the subscript when there is no misunderstanding in the context. For a set $\mathcal X$, we denote the free group generated by the set $\mathcal X$ as $F(\mathcal X)$. We also use $F(\mathcal X)$ to represent the set of reduced words in alphabet $\mathcal X\cup \mathcal X^{-1}$. 

In addition, for a group $G$ and a commutative ring $K$ with $1\neq 0$, we let $KG$ be the group ring of $G$ over $K$. An element $\lambda\in KG$ is usually denoted as $\lambda=\sum_{g\in G} \alpha_g g, \alpha_g\in K$ where all but finitely many $\alpha_g$'s are 0. We also regard $\lambda$ as a function $\lambda: G\to K$ with finite support, where $\lambda(g)=\alpha_g$. 

We say a group $G$ is an \emph{extension} of a group $A$ by a group $T$ if $A$ is a normal subgroup of $G$ and $T\cong G/T$. If $A$ is abelian, then $A$ is a module over $\mathbb ZT$ and the action of $T$ on $A$ is given by conjugation. In this case, we also say that $G$ is an extension of a $T$-module $A$ by $T$. 

\subsection{Dehn Function}
\label{prelim2}

Let $G$ be a finitely presented group in the category of all groups, namely $G=\langle X\mid R \rangle$ where $|X|,|R|<\infty$. We denote this presentation by $\mathcal P$. The length of a word $w\in G$ is the length of the corresponding reduced word in $F(X)$. A word $w$ in the alphabet $X\cup X^{-1}$ equals 1 if and only if it lies in the normal closure of $R$, i.e.
\[w=_{F(X)} \prod_{i=1}^k r_i^{f_i} \text{ where } r_i\in R\cup R^{-1}, f_i\in F(X).\]
The smallest possible $k$ is denoted by $\area_\mathcal P(w)$. The \emph{Dehn function} of $G$ with respect to the presentation $\mathcal P$ is defined as 
\[\delta_{\mathcal P}(n)=\sup\{\area_{\mathcal P}(w)\mid |w|_{F(X)}\leqslant n\}.\] 

It is convenient for us to talk about functions up to the equivalence relation $\approx$ which we defined in the introduction. Because the Dehn function is independent of finite presentations of a group in terms of $\approx$. That is,

\begin{proposition}
\label{presentationEquivalence}
Let $\mathcal P,\mathcal P'$ be two finite presentations of $G$, then 
\[\delta_{\mathcal P}(n)\approx \delta_{\mathcal P'}(n).\]
\end{proposition}

We shall then denote the Dehn function of a finitely presented group $G$ by $\delta_G(n)$. Moreover, the Dehn function is a quasi-isometric invariant up to $\approx$. 

\begin{theorem}[\cite{gromov1996geometric}]
\label{quasi}
Let $G$ be a finitely presented group, $G=\langle X \rangle$. Let $H$ be a finitely generated group with generating set $Y$, $|X|,|Y|<\infty$. If $G,H$ are quasi-isometric, then $H$ is finitely presented and $\delta_{G,X}\approx \delta_{H,Y}$.
\end{theorem} 

\cref{presentationEquivalence} follows immediately from the fact that Cayley graphs of the same group over different generating sets are quasi-isometric. 

Here is another useful consequence of \cref{quasi}. 
\begin{corollary}
\label{finiteindex}
	Let $G$ be a finitely presented group and $H\leqslant G$ such that $[G:H]<\infty$, then $H$ is finitely presented and quasi-isometric to $G$ hence $\delta_H\approx \delta_G$. 
\end{corollary}

\subsection{Van Kampen Diagrams}
\label{perlim3}

One way to visualize the area of a given word is to consider what is called a \emph{van Kampen diagram}. Let $G=\langle X\mid R\rangle$ be a finitely presented group and $w$ be a reduced word which is equal to 1. Then by the previous discussion, $w$ has a decomposition as following:
\begin{equation}
\label{decomposition}
	w=_{F(X)}  \prod_{i=1}^k r_i^{f_i} \text{ where } r_i\in R\cup R^{-1}, f_i\in F(X).
\end{equation}

For every decomposition (\ref{decomposition}), we can draw a diagram which consists of a bouquet of ``lollipops". Each ``lollipop" corresponds to a factor $r_i^{f_i}$, the stem of which is a path labeled by $f_i$ and the candy of which is a cycle path labeled by $r_i$. Going counterclockwise around the ``lollipop" starting and ending at the tip of the stem, we read $f_i^{-1}r_i f_i$. Thus the boundary of the bouquet of ``lollipops" is labeled by the word which is the right-hand side of (\ref{decomposition}). 

Note that we obtained $w$ from the right hand side of (\ref{decomposition}) by cancelling all consecutive pairs of $xx^{-1}$ or $x^{-1}x, x\in X$ on the boundary and removing subgraphs whose boundaries labelled by $xx^{-1}$ or $x^{-1}x, x\in X$ (which is a ``dipole" or a sphere). In the diagram, the corresponding process is identifying two consecutive edges with the same label but different orientation on the boundary. After finitely many such reductions, we will obtain a diagram whose boundary is labeled by $w$. 

\begin{figure}[H]
		\centering
			\includegraphics[width=13cm]{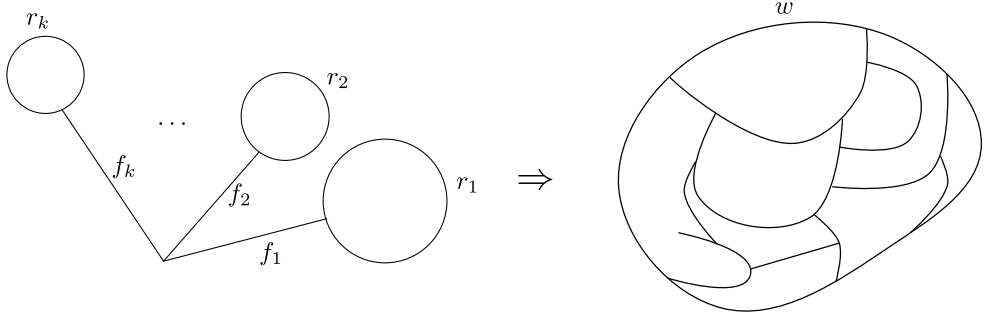}
			\caption{A bouquet of ``Lolipops" and its corresponding van Kampen Diagram}
\end{figure}

The resulting diagram is the Van-Kampen diagram of $w$. The edges are labeled by elements in $X$ and cells are (i.e. the closure of a bounded connected components of the plane minus the graph) labeled by words from $R\cup R^{-1}$. 

For example, in group $\langle a,b\mid [a,b]=1\rangle$, the Van-Kampen diagram of $[a^2,b]=[a,b]^a[a,b]$ looks like this.

\begin{figure}[H]
		\centering
			\includegraphics[width=13cm]{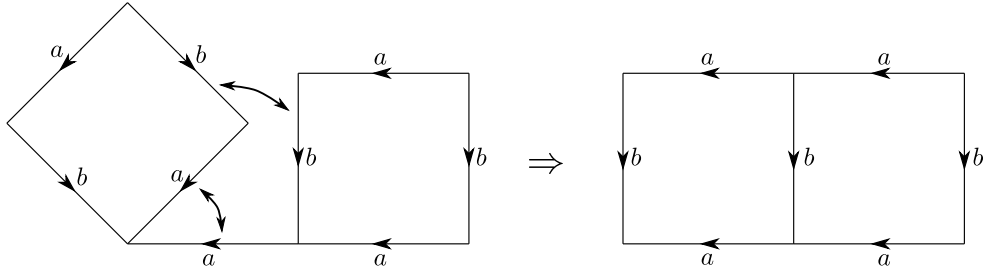}
			\caption{Obtain the van Kampen diagram of $[a^2,b]$}
\end{figure}

The following is called the van Kampen Lemma.

\begin{lemma}
	If a reduced group word $w$ over the alphabet $X$ is equal to 1 in $G=\langle X\mid R\rangle$, then there exists a van Kampen diagram over the presentation of $G$ with boundary label $w$. 
	
	Conversely, let $\Delta$ be a van Kampen diagram over $G=\langle X\mid R\rangle$ where $X=X^{-1}$ and $R$ is closed under cyclic shifts and inverses. Let $w$ be the boundary of $\Delta$. Then $w$ is equal in the free group $F(X)$ to a word of the form $u_1r_1u_2r_2\dots u_kr_ku_{k+1}$ where 
	\begin{enumerate}[(1)]
		\item each $r_i$ is from $R$;
		\item $u_1u_2\dots u_{k+1}=1$ in $F(X)$;
		\item $\sum_{i=1}^{m+1}|u_i|\leqslant 4e$ where $e$ is the number of edges of $\Delta$.
	\end{enumerate}
	In particular, $w$ is equal to 1 in $G$.
\end{lemma}

We say a Van Kampen diagram is \emph{minimal} if it has the minimal number of cells over all such diagrams of the same word. For a word $w=_G 1$, the area of $w$ is the same as the number of cells of a minimal van Kampen diagram. This fact is useful to estimate the lower bound of the Dehn function. We will use the van Kampen diagram to compute Dehn functions of a class of examples in the last section. 

Other applications of the van Kampen diagram can be found in many books. For example, in the book \cite{ol2012geometry}, one can found the study of van Kampen diagrams used to construct groups with extreme properties such as infinite bounded torsion group, Tarski monsters, etc.

\section{The Membership Problem of a Submodule of a Free Module}
\label{membership}

\subsection{A Well-Order on the Free Module}
\label{membership1}

Let $R:=\mathbb Z[x_1,\dots,x_k]$ be a polynomial ring over $\mathbb Z$. Consider a free $R$-module $M$ with basis elements $e_1,\dots,e_m$. A \emph{term} in $M$ is a product of an integer, a monomial in $R$, and an element from the basis. A typical term looks like $a\mu e_i$, where $a\in \mathbb Z, \mu$ is a monomial in $R$. Let $\mathcal T$ to be the set of all terms in $M$. In addition, we will call $\mu e_i$ a \emph{monomial} in $M$, denoted  by $\mathcal U$ the set of monomials of $M$. The set of monomials in the polynomial ring $R$ in the usual sense will be denoted by $\mathcal X$. For a term $g\in \mathcal T$, we denote $C(g), M(g)$ to be the coefficient, and monomial part of $g$ respectively. An element in $M$ is a finite sum of terms. From now on, we only consider reduced elements in $M$, in a sense that no terms are sharing the same monomial. We also denote $\supp(f)$ to be the set of monomials with non-zero coefficients.

Our first goal is to put a well-order on $\mathcal T$. To construct such an order, we have to put well-orders on $\mathbb Z,\mathcal X$ and $\{e_1,\dots,e_m\}$ separately. Then we will extend it to $M$. 

On $\mathbb Z$, we define an order $\prec$ as following
\[0\prec 1 \prec 2 \prec \dots \prec -1 \prec -2 \prec \dots.\]
Under this order, all negative numbers are larger than any positive number. Let $a,b\in \mathbb Z$ where $a\prec b$, then there exists unique $q,r$ such that $a=qb+r, 0< r < |b|$. Note that $r\prec a$ whether $a$ is positive or negative, thus we can reduce any number to its remainder in this sense. The remainder is always positive, which is important to us. It is not hard to see that $\prec$ on $\mathbb Z$ is a well-order.

For monomials in $R$, we use the degree lexicographical order (also called shortlex or graded lexicographical order) $\prec$ which is defined with respect to $x_1 \succ x_2 \succ \dots \succ x_k$, i.e. for $\mu_1=x_1^{n_1}x_2^{n_2}\dots x_k^{n_k}, \mu_2=x_1^{m_1}x_2^{m_2}\dots x_k^{m_k}$
\[\mu_1\prec \mu_2 \text{ if }\sum_{i=1}^k|n_i|<\sum_{i=1}^k |m_i| \text{ or } \sum_{i=1}^k|n_i|=\sum_{i=1}^k |m_i|, \mu_1\prec_{lex} \mu_2,\]
where $\prec_{lex}$ is the usual lexicographical order which in defined in the following way
\[x_1^{n_1}x_2^{n_2}\dots x_k^{n_k}\prec_{lex} x_1^{m_1}x_2^{m_2}\dots x_k^{m_k} \text{ if $n_i<m_i$ for the first $i$ where $n_i$ and $m_i$ differ}.\]
$\prec$ on $\mathcal X$ in fact is a well-oder while $\prec_{lex}$ might not be. (See \cite{baader1999term}.)

Finally we fix an order $e_1\succ e_2 \succ \dots \succ e_s$. We now set $\prec$ on $T$ to be the lexicographical order based on $ \mathcal X \succ \{e_1,\dots,e_m\} \succ \mathbb Z$. For instance, 
\[7x_1^2x_2e_2 \prec 5x_1^{3}e_1, 3x_1^3x_2^5e_2\prec 3x_1^3x_3^6e_2, 2x_1^5x_3^2e_3\prec 4 x_1^5x_3^2e_3. \]
Since $\prec$ on each of $\mathcal X, \{e_1,\dots,e_m\} , \mathbb Z$ is a well-order, we obtain a well-order $\prec$ on $\mathcal T$. 
 
With the well-order $\prec$, we are able to compare any two terms. Consequently, we can define the leading monomial $LM(f)$ of $f$ to be the largest monomial among $\supp(f)$. For example, 
\[LM(x_1^7e_1+3x_1^3x_2^4e_2)=x_1^7e_1,LM(x_2^3e_1+(x_2^5x_3^2+x_2^3x_4^5)e_2+x_2^5x_3^2e_3)=x_2^3x_4^5e_2.\]
Next, we define the leading coefficient of $f$ to be the coefficient of the leading monomial, denoted by $LC(f)$. Then the leading term of $f$ can be defined as
\[LT(f):=LC(f)\cdot LM(f).\] 

We then extend $\prec$ to $M$. For $g,f\in M$, we define $g\prec f$ inductively as follows
\[g\prec f \text{ if } LT(g) \prec LT(f) \text{ or } LT(g)=LT(f), g-LT(g)\prec f-LT(f).\]
Since $\prec$ on $T$ is a well-oder, then so is $\prec$ on $M$.

Note that $\prec$ is compatible with multiplication by elements from $\mathcal X$ i.e., if $g\prec h$, then $\mu g\prec \mu f$, $\mu \in \mathcal X$.

One remark on $\prec$ is that it is Noetherian on $\mathcal U$ as well as $\mathcal X$, the set of monomials in $M$, i.e., there is no infinite descending chain of monomials. However, the statement is not true for $\prec$ on $\mathcal T$. Because we have an infinite descending chain for negative numbers. This issue can be avoided by what we will introduce in the next subsection: polynomial reduction.

\subsection{Polynomial Reduction and Gr\"{o}bner Bases}
\label{membership2}

Now let us define the key ingredient for the application of Gr\"{o}bner bases: polynomial reduction.

For two monomials $\mu e$ and $\mu' e'$ from $\mathcal U$, we say $\mu e\mid \mu' e'$ if $\mu\mid \mu'$ and $e=e'$. Let $F=\{f_1,\dots,f_l\}$ be a finite subset of $M$ and $S$ be the submodule generated by $F$. Given $g,h\in M$, we define the polynomial reduction $g \rightarrow_{F} h$ as follows: if there exists $f\in F$ and a term $g_0\in \mathcal T$ of $g$ such that $LM(f)\mid M(g_0), LC(f)\prec C(g_0)$, then
\[g=\frac{qM(g_0)}{LM(f)} f+h,\]
where $C(g_0)=q LC(f)+r$, $q,r$ are unique integers such that $0< r <|LC(f)|.$
 
 For $g \longrightarrow_{F} h$, read ``$g$ reduces to $h$ modulo $F$''. If there's no such $f$ and $g_0$, then we say that $g$ is \emph{irreducible} modulo $F$.

 Note that we naturally have $h\prec g$ if $g\longrightarrow_F h$. We claim that $\longrightarrow_F$ is Noetherian, i.e., there is no infinite reduction sequence. First, note that we turn the coefficient of $M(g_0)$ of $h$ to a positive number after a reduction, then there are only finitely many possible reductions that can be applied to monomial $M(g_0)$. Thus if we assume that $g_0$ is the largest term can be reduced in $g$ modulo $F$, then after finitely many reductions, the monomial of the largest term that can be reduced is strictly less than the original one. Since $\prec$ is Noetherian for monomials, we only have a reduction of finite length for any given $g\in M$. 
 
 Let $\longrightarrow^*_F$ be the reflexive and transitive closure of $\longrightarrow_F$. Then for each $g\in M$, there exists $h\in M$ such that $g\longrightarrow_F^* h$ and $h$ is irreducible modulo $F$. We call $h$ to be a reduced form of $g$ modulo $F$. Unfortunately, the reduced form of an element in $M$ may not be unique. In fact, at each step of reduction, we may have multiple choices of $f\in F$ that can be applied to this reduction. This yields our motivation for defining \emph{Gr\"{o}bner basis}: a generating set $F$ such that every element in $M$ has a unique reduced form modulo $F$. In theoretical computer science, this property is called Church-Rosser property (See \cite{buchberger1985algorithmic}).
 
 We denote $g\equiv_S h$ if $g-h\in S$. $\equiv_S$ defines an equivalent relation on $M$. We let the \emph{normal form} of $g$ to be the smallest element in its equivalent class with respect to $\prec$. It is well-defined since $\prec$ is a well-order. We denote the normal form of $g$ by $NF(g)$.

\begin{definition}
	Let $M$ be a free $R$-module of finite rank, and $S$ be a submodule of $M$. A finite generating set $F$ of $S$ is called \emph{Gr\"{o}bner basis} if $g \rightarrow_F^* NF(g),\forall g\in M$.
\end{definition}

\begin{rems}
	$R$ is a Noetherian ring hence $M$ is a Noetherian module. Thus any submodule of $M$ is finitely generated.
\end{rems}
	
Let $S_u=\{g\in S| LM(g)=u\},u\in \mathcal U$ and $L_u=\{LC(g)\mid g\in S_u\}$. It is not hard to see that $L_u$ is an ideal in $\mathbb Z$. Thus it is generated by the smallest element in this ideal with respect to $\prec$. We denote $h_u$ to be element such that $LC(h_u)$ generates $L_u$ since $\mathbb Z$ is a principle ideal domain. Note that the leading coefficient of $h_u$ is always positive. Since by our definition of $\prec$, negative numbers are larger than positive numbers. For our purpose, we denoted it $c_u$ hence $LT(h_u)=c_uu$. Let $P$ be the set of all such $h_u$ which generates $S$ over $\mathbb Z$ whenever $h_u$ can be defined (since $S_u$ might be empty). 

\begin{theorem}[\cite{sims1994computation}, Proposition 10.6.3]
	For any submodule of a free module of finite rank over $R=\mathbb Z[x_1,\dots,x_k]$, there exists a Gr\"{o}bner basis.
\end{theorem}

Let $F$ be a Gr\"{o}bner basis for $S$. According to the proof of (\cite{sims1994computation}, Proposition 10.6.3), $F$ can be constructed as a finite set such that $P=\mathcal XF:=\{uf\mid u\in \mathcal X,f\in F\}$, which allows us to reduce $g\in S$ by deleting the leading terms with elements in $F$. In the next subsection, we will use this specific $F$ for the Gr\"{o}bner basis.

\subsection{Division Algorithm}
\label{membership3}

For an element $g\in M$, $g$ can be written as a finite sum of distinct terms, i.e
\[g=c_1u_1+c_2u_2+\dots+c_d u_d,\]
where $c_i\in \mathbb Z, u_i\in \mathcal U$ and $u_1\succ u_2\succ \dots \succ u_d$. We define the length of $g$ to be $|g|:=\sum_{i=1}^d |c_i|$. And if the leading monomial of $g$ is $x_1^{n_1}x_2^{n_2}\dots x_k^{n_k}e_i$, we define $\deg(g)=\sum_{i=1}^k n_i$. One immediate observation is that if $g\prec h$ then $\deg(g)\leqslant \deg(h)$.

Let $F=\{f_1,\dots,f_l\}$ be a Gr\"{o}bner basis for a submodule $S$ and $g=c_1u_1+\dots+c_d u_d\in S$ such that $\deg(g)\leqslant n, |g|\leqslant p$. Since $g\in S$, then $g\longrightarrow_F^* 0$. Thus there exists a finite sequence of reductions
\[g=g_0\longrightarrow_F g_1 \longrightarrow_F g_2 \longrightarrow_F g_3\longrightarrow_F \dots \longrightarrow_F g_r=0.\]
At each step, if we always choose to cancel the leading term of $g_i$ using the polynomial reduction (this is always possible since $g$ can be reduced to 0), we may assume that $LM(g_0)\succ LM(g_1)\succ LM(g_2)\succ \dots \succ LM(g_r)=0$. Thus the number of steps of reduction is bounded by the number of monomials less or equal to $LM(g)$. Recall that $m$ is the rank of the free module, then 
\[r\leqslant |\{u\in \mathcal U\mid u\prec LM(g)\}|\leqslant mG_k(n),\]
where $G_k(n)$ is the growth function of a free commutative monoid with a free generating set of size $k$ (See \cite{sapir2014combinatorial,3.7.1}). It is well-known that $G_k(n)$ is a polynomial of degree $k$. In fact, 
\[G_k(n)={n+k \choose k}. \]

At the $j$th step of reduction, we have
\[g_{j}=g_{j-1}-a_{j}\mu_j f_{i_j},\]
where $a_j\in \mathbb Z, \mu_j\in \mathcal X, 1\leqslant i_j\leqslant l$ and $LT(g_{j-1})=LT(a_{j}\mu_i f_{i_j})$. Then $|a_j|\leqslant LC(g_{j-1})\leqslant |g_{j-1}|$. Let $C=\max\{|f_1|,|f_2|,\dots,|f_l|\}.$ We also observe that 
\begin{equation}
\label{eq01}
|g_{j}|\leqslant |g_{j-1}|+|a_j||f_{i_j}|\leqslant |g_{j-1}|+C|a_j|.
\end{equation}
Additionally, we have $|a_1|\leqslant |g_0|=p$, and 
\begin{equation}
\label{eq02}
|a_{j}|\leqslant LC(g_{j-1})\leqslant |g_{j-1}|.
\end{equation}
Combine (\ref{eq01}) and (\ref{eq02}) inductively,  		
\[|a_{j}|\leqslant |g_{i-1}|\leqslant |g_{i-2}|+C|a_{j-1}|\leqslant |g_{i-2}|(1+C)\leqslant \dots \leqslant p(1+C)^{j-1}, j\geqslant 1.\]
Adding all the steps up, we have
\[g=\sum_{j=1}^r a_{j}\mu_j f_{i_j}=\sum_{i=1}^l \alpha_i f_i, \alpha_i\in R. \]
Note that 
\[\sum_{i=1}^l|\alpha_i|=\sum_{j=1}^r |a_j|\leqslant p(1+(1+C)+(1+C)^2+\dots+(1+C)^{r-1})\leqslant \frac{p((1+C)^{mG_k(n)}-1)}{C}.\]

In general, one important consequence of Gr\"obner bases is the following

\begin{corollary}[Division]
\label{division0}
	Let $M$ be a free module over a polynomial ring $R=\mathbb Z[x_1,\dots,x_k]$. Let $F=\{f_1,\dots,f_l\}$ be a Gr\"obner basis for a submodule $S$. Then there exists a constant $K$ such that for every $g\in M,\deg(g)\leqslant n,|g|\leqslant p$, one can write
	\[g=\sum_{i=1}^n \alpha_i f_i+r\]
with $\alpha_i\in R,r=NF(g)$ and 
\[\deg(\alpha_i f_i)\prec \deg(g), 1\leqslant i \leqslant l, \sum_{i=1}^l|\alpha_i|\leqslant pK^{n^k}.\]
\end{corollary}
\begin{rems}
	This provides an algorithm to solve the membership problem for submodules of a finitely generated free module over polynomial rings with integral coefficients. Given $g,f_1,\dots,f_l$, to decide if $g$ lies in the submodule $S$ generated by $f_1,\dots,f_l$ we first find a Gr\"{o}bner basis for $S$. The algorithm which finds Gr\"{o}bner bases can be found in \cite{sims1994computation}. Once we have Gr\"{o}bner bases in hand, we can compute the normal form of $g$. Last, we use the fact that $g\in S$ if and only if $NF(g)=0$.
	\end{rems}

Let $T$ to be the free abelian group of rank $k$ with basis $t_1,\dots,t_k$. We can regard the group ring $\mathbb ZT$ as a factor ring of $\mathbb Z[t_1,t_1^{-1}\dots,t_k,t_k^{-1}]$ i.e
\[\mathbb ZT\cong \mathbb Z[t_1,t_1^{-1}\dots,t_k,t_k^{-1}]/\langle t_1t_1^{-1}-1,\dots,t_kt_k^{-1}-1\rangle.\]
Then a submodule generated by a finite set $F$ over $\mathbb ZT$ can be identified as a submodule generated by $F\cup\{t_1t_1^{-1}-1,\dots,t_kt_k^{-1}-1\}$ over $\mathbb Z[t_1,t_1^{-1}\dots,t_k,t_k^{-1}]$.

Therefore we have a similar result for group rings
\begin{corollary}
\label{division}
	Let $M$ be a free module over $\mathbb ZT$ where $T$ is a free abelian group of rank $k$. $S$ is a submodule of $M$. Then there exists a finite generating set $F=\{f_1,\dots,f_l\}$ and a constant $K$ such that for $g\in S$ with $\deg(g)\leqslant n,|g|\leqslant p$ there exists $\alpha_1,\dots,\alpha_l\in \mathbb ZT$
	\[g=\alpha_1f_1+\dots+\alpha_l f_l, \deg(\alpha_if_i)\leqslant \deg(g),\sum_{i=1}^l |\alpha_i|\leqslant pK^{n^{2k}}.\]  
\end{corollary}

\begin{rems}
$\deg (g)$ and $|g|$ for element $g\in \mathbb ZT$ are inherited from $\mathbb Z[t_1,t_1^{-1}\dots,t_k,t_k^{-1}]$.		
\end{rems}

\section{A Characterization of Finitely Presented Metabelian Groups}
\label{finitepresented}

\subsection{A Geometric Lemma}
\label{finitepresented1}

Let $\mathbb R^n$ be the Euclidean vector space with the usual inner product $\langle \cdot,\cdot\rangle$. We denote the norm induced by this inner product by $\|x\|=\sqrt{\langle x,x\rangle}$. If $r>0$, then $B_r$ denotes the open ball of radius $r$, i.e $B_r=\{x\in \mathbb R^n\mid \|x\|<r\}$.

We consider a \emph{finite} collection $\mathcal F$ of \emph{finite} subsets $L\subset \mathbb R^n$. Say that an element $x\in \mathbb R^n$ \emph{can be taken from} $B_r$ if either $x\in B_r$ or if there exists $L\in \mathcal F$ such that 
\[x+L=\{x+y\mid y\in L\}\subset B_r.\]

\begin{lemma}[\cite{bieri1980valuations}, Lemma 1.1]
\label{geometric}
	Assume that for every $0\neq x\in \mathbb R^n$, there is $L\in \mathcal F$ such that $\langle x,y\rangle>0$ for all $y\in L$. Then there exists a radius $r_0\in \mathbb R^{+}$ and a function $\varepsilon: (r_0,\infty)\to \mathbb R^+$ with the property that for $r>r_0$ each element of $B_{r+\varepsilon(r)}$ can be taken from $B_r$ by $\mathcal F$. 
\end{lemma}

We omit the proof of this lemma but we will use the explicit choice of $r_0$ and $\varepsilon: (r_0,\infty)\to \mathbb R^+$ in the proof of (\cite{bieri1980valuations}, Lemma 1.1). Let $S^{n-1}$ be the unit sphere in $\mathbb R^n$ and consider the function $f:S^{n-1}\to \mathbb R$ given by 
\[f(u)=\max_L \min_y \{\langle u,y\rangle \mid y\in L \in \mathcal F\}, \text{ for } u\in S^{n-1}.\]
The function $f$ is continuous. By the assumption on $\mathcal F$, we have $f(u)>0$ for all $u\in \mathcal S^{n-1}$. Since $S^{n-1}$ is compact, we can define 
\[C=\inf\{f(u)\mid u\in S^{n-1}\}>0, D=\max_L \min_y \{\|y\|\mid y\in L\in \mathcal F\}>0.\]

Then our choice of $r_0$ and $\varepsilon$ are
\[r_0=\frac{D^2}{2C}, \varepsilon(r)=C-\frac{D^2}{2r}.\]

\subsection{A Theorem by R.Bieri and R. Strebel}
\label{finitepresented2}

Let $T$ be a finitely generated abelian group, written multiplicatively. A (real) \emph{character} of $T$ is a homomorphism $\chi: T\to \mathbb R$ of $T$ into the additive group of the field of real numbers $\mathbb R$.  Let $\tor T$ be the torsion subgroup of $T$. Then $T/\tor T\cong \mathbb Z^k\subset \mathbb R^k$ where $k$ is the rank of $T$. We fix a homomorphism $\theta: T\to \mathbb R^k$. For every character $\chi:T\to \mathbb R$, there is a unique $\mathbb R$-linear map $\bar \chi: \mathbb R^k\to \mathbb R$ such that $\chi=\bar\chi \circ \theta $. And by the Riesz representation theorem, there is a unique element $x_\chi\in \mathbb R^k$ such that $\bar\chi (y)=\langle x_\chi,y\rangle, \forall y\in \mathbb R^k$, whence $\chi(t)=\langle x_\chi,\theta(t)\rangle$ \cite{bieri1980valuations}. Therefore each character $\chi$ corresponds a vector $x_\chi$ in $\mathbb R^k$. Conversely, given a vector $x$ in $\mathbb R^k$, we can define a corresponding character by $\chi(t)=\langle x, \theta(t)\rangle$. This will be a useful realization for characters on $T$.

Every character $\chi: T\to \mathbb R$ can be extended to a ``character'' of the group ring $\chi: \mathbb ZT\to \mathbb R\cup\{+\infty\}$ by putting $\chi(0)=+\infty$ and 
\[\chi(\lambda)=\min\{\chi(t)\mid t\in \supp(\lambda)\}, \text{ where }0\neq \lambda\in \mathbb ZT.\]

One can check that $\chi(\lambda\mu)\geqslant \chi(\lambda)+\chi(\mu)$ for all $\lambda, \mu\in \mathbb ZT$. Moreover, if $T$ is free abelian, the group ring $\mathbb ZT$ has no zero divisor. It follows that $\chi(\lambda\mu)= \chi(\lambda)+\chi(\mu)$ in this case \cite{bieri1980valuations}. 

For every $T$-module $A$, the \emph{centralizer} $C(A)$ of $A$ is defined to be 
\[C(A)=\{\lambda \in \mathbb ZT\mid \lambda \cdot a=a, \forall a\in A\}.\]

If $A$ is a left (right) $T$-module then we write $A^*$ for the right (resp. left) $T$-module with $T$-action given by $at=t^{-1}a$ (resp. $ta=at^{-1}$). 

We say a $T$-module $A$ is \emph{tame} if $A$ is finitely generated as a $T$-module and there is a finite subset $\Lambda\subset C(A)\cup C(A^*)$ such that for every non-trivial character $\chi:\mathbb ZT\to \mathbb R$ there is $\lambda\in \Lambda$ with $\chi(\lambda)>0$. 

Robert Bieri and Ralph Strebel proved the following theorem, that characterizes finitely presented metabelian groups.
\begin{theorem}[\cite{bieri1980valuations}, Theorem 5.1]
\label{fp}
    Let $G$ be a finitely generated group and let $A\triangleleft G$ be a normal subgroup such that both $A$ and $T=G/A$ are abelian. Then $G$ is finitely presented if and only if $A$ is tame as a $T$-module.
\end{theorem}

For our purpose, let us sketch the proof of the ``if" part of this theorem, more precisely

\begin{theorem}[\cite{bieri1980valuations}, Theorem 3.1]
\label{OD}
	If $T$ is a finitely generated abelian group and $A$ is a tame $T$-module, then every extension of $A$ by $T$ is finitely presented.
\end{theorem}

To prove \cref{OD}, we have to introduce some preliminary concepts in order to provide a reasonable sketch. We first define ordered and semi-ordered words. Let $F$ be the free group freely generated by $\mathcal T=\{t_1,\dots,t_k\}$. Let $\bar F\subset F$ denote the subset of all \emph{ordered} words of $F$, i.e.
\[\bar F=\{t_1^{m_1}t_2^{m_2}\dots t_k^{m_k}\mid m_1,\dots,m_k\in \mathbb Z\}.\]
If $w\in F$, we write $\bar w$ as the unique word from $\bar F$ representing $w$ modulo the derived subgroup $F'$. In addition, a word $w\in F$ is said to be \emph{semi-ordered} if it is of the form
\[w=t_{\sigma(1)}^{m_1}t_{\sigma(2)}^{m_2}\dots t_{\sigma(k)}^{m_k}\]
where $\sigma$ is a permutation of the symbols $\{1,\dots,k\}$. 

Let $\theta: F\to \mathbb R^k$ be the homomorphism given by 
\[\theta(t_i)=(\delta_{i1},\dots,\delta_{ik}),\]
for $1\leqslant i \leqslant k$. For every $w\in F$ define the \emph{trace} $\Tr{w}\subset \mathbb R^n$ as follows: if 
\[w=s_1s_2\dots s_m, \text{ where } s_j\in \mathcal T\cup \mathcal T^{-1},\]
is freely reduced, then 
\[\Tr(w)=\{\theta(s_1\dots s_j)\mid j=0,1,\dots, m\},\] 

Next, we define a sequence of auxiliary groups. Let $\mathcal A$ be a finite set and choose an assignment picking an element $a_{ij}\in \mathcal A$ for every pair of integers $(i,j)$ with $1\leqslant i<j \leqslant k.$ For every $r\in \mathbb R^+\cup \{+\infty\}$, let $H_r$ be the group generated by the set $\mathcal A\cup \mathcal T$ with the following defining relations.
\begin{align}
	&[t_i,t_j]=a_{ij}, \text{ for } 1\leqslant i<j\leqslant k, \\
	&[a,b^u]=1, \text{ for } a,b\in \mathcal A, u\in \bar F \text{ with } \|\theta(u)\|<r.
\end{align}

We have some useful properties for the group $H_r$
\begin{proposition}
\label{HP}
	If $r\in \mathbb R^+$, then 
	\begin{enumerate}[(a)]
		\item (\cite{bieri1980valuations}, Lemma 3.2) $a^{\bar w}=a^w$ for every $a\in \mathcal A$ and every $w\in F$ with $\Tr(w)\subset B_r$.
		\item (\cite{bieri1980valuations}, Lemma 3.4a) For $u,v\in F$ such that 
        \[\Tr(u)\subset B_r,\Tr(v)\subset B_r,\|\theta(uv)\|<r.\]
        Then $[a,b^{uv}]$ and $[a,b^{\overline{uv}}]$ are conjugate in $H_r$ for every $a,b\in \mathcal A$.
        \item (\cite{bieri1980valuations}, Lemma 3.4b) Assume $r>2k$. Let $u,v$ be semi-order words in $F$ such that 
        \[\|\theta(u)\|\leqslant \frac{r}{2k},\|\theta(v)\|\leqslant r+\frac{1}{2k},\|\theta(uv)\|<r\]
        are satisfied. Then $[a,b^{uv}]$ and $[a,b^{\overline{uv}}]$ are conjugate in $H_r$ for every $a,b\in \mathcal A$.
	\end{enumerate}
\end{proposition}
	
For $r=\infty$, $H_\infty$ is metabelian and it is an extension of a free abelian group by another free abelian group. In fact, by the definition of $H_\infty$ and \cref{HP} (a), we have $[a^u,b^v]=1$ for any $u,v\in F, a,b\in \mathcal A$ and $a^u$ is of infinite order. Then $\llangle \mathcal A \rrangle_{H_\infty}$ is free abelian of infinite rank with basis $\{a^u\mid a\in \mathcal A,u\in \bar F\}$. $H_\infty/\llangle \mathcal A\rrangle$ is generated by $\mathcal T$. It is abelian since we includes all commutators $[t_i,t_j]$ in $\mathcal A$. Each $t_i$ is of infinite order. It follows that $H_\infty/\llangle \mathcal A\rrangle$ is also free abelian. But let us emphasis this: $H_\infty$ is infinitely related.

Now back to the proof of \cref{OD}. We first claim that the problem can be reduced to the case when $T$ is a free abelian group. Let $\pi: G\to T$ be the epimorphism and $T_1\leqslant $ T be a complement of the torsion subgroup of $T$. Then $G_1=\pi^{-1}(T_1)$ has finite index in $G$ and $G_1$ is an extension of an abelian group by a finitely generated free abelian group. $G$ is finitely presented if and only if $G_1$ is finitely presented. Moreover, if $A$ is a tame $T$-module, then $A$ is also a tame $T_1$-module (\cite{bieri1980valuations}, Prop. 2.5). Therefore, the statement of \cref{OD} is true for $G$ if and only if it is true for $G_1$.

Now we assume that $T$ is a free abelian group of rank $k$, $A$ is a tame $T$-module, and $G$ is an extension of $A$ by $T$. Denote $\pi: G \twoheadrightarrow T$ to be the epimorphism such that $A\cong \ker \pi$. 

Let $\mathcal T=\{t_1,\dots,t_k\}$ be a subset of $G$ such that $\{\pi(t_1),\dots, \pi(t_k)\}$ forms a basis of $T$ and $\mathcal A$ be a finite subset of $A$ containing all commutators $a_{ij}=[t_i,t_j]$ for $1\leqslant i<j\leqslant k$ and generating $A$ as a $T$-module. We write $\hat w\in T$ for the image of $w\in F$ under $\pi$.

Since $A$ is a tame $T$-module. Then there is a finite subset $\Lambda\subset C(A)\cup C(A^*)$ with the property that for every character $\chi: T\to \mathbb R$, there exists $\lambda\in \Lambda$ such that $\chi(\lambda)>0$. Recall that $F:=F(\mathcal T)$ and $\bar F$ is the set of ordered words of $F$. For every $r\in (0,+\infty]$, we define the group $G_r$ to be given by generators $\mathcal A\cup \mathcal T$ and defining relations
\begin{align}
	[t_i,t_j]&=a_{ij}, &\text{for } 1\leqslant i<j\leqslant k,\label{commutative1}\\
	[a,b^u]&=1, &\text{for $a,b\in \mathcal A,u\in \bar F$ with $\|\theta(u)\|<r$},\label{commutative2}\\
	\prod_{u\in \bar F} (a^{\lambda(\hat u)})^u&=a, &\text{for } a\in \mathcal A, \lambda\in \Lambda\cap C(A), \label{action1}\\
	\prod_{u\in \bar F} (a^{\lambda(\hat u)})^{u^{-1}}&=a, &\text{for } a\in \mathcal A, \lambda\in \Lambda\cap C(A^*). \label{action2}
\end{align}
In relations (\ref{action1}) and (\ref{action2}), we regard $\lambda$ as a finite supported function from $T$ to $\mathbb Z$. Hence $\lambda(\hat u)$ is just the value of $\lambda$ at $\hat u$.

$G_r$ is finitely presented if $r\neq +\infty$. If $r=\infty$, although $G_\infty$ is not finitely presented, it is metabelian once we realize $G_\infty$ is a factor group of $H_\infty$.
For each $\lambda\in \Lambda$, $\theta(\supp(\lambda))$ is a finite subset of $\mathbb R^k$, denoted by $L_\lambda$. Let $\mathcal F=\{L_\lambda \mid \lambda\in \Lambda\}$. As previous discussion, there is a one-to-one correspondence between each character $\chi:T\to \mathbb R$ and a linear functional $\langle v_\chi, \cdot \rangle$. Therefore if $A$ is tame, $\mathcal F$ is a collection of finite sets which satisfies assumptions of \cref{geometric}.

Let
\[C=\inf_{u\in S^{n-1}} \max_{\lambda\in \Lambda} \min_{y\in L_\lambda}\{\langle u,y\rangle\},D=\max_{\lambda\in \Lambda}\min_{y\in L_\lambda}\{\|y\|\}.\]
In addition, let $R=2k\max\{D,D^2/2C\}$. We have the following lemma
\begin{lemma}[\cite{bieri1980valuations}, Lemma 3.5]
\label{finiteness}
For $r\in [R,\infty)\cup \{\infty\}$, $G_r\cong G_R$. In particular, $G_\infty$ is finitely presented. 
\end{lemma}

Since relations (\ref{commutative1})-(\ref{action2}) hold in $G$, then $G$ is a factor group of $G_\infty$. The epimorphism $\varphi: G_\infty \to G$ is induced by the identity map on $\mathcal A\cup \mathcal T$. By the fact that the normal subgroup of a finitely generated metabelian group is normal closure of a finite set \cite{Hall1954}, $G$ is finitely presented. Thus we finished the proof of \cref{OD}.
 
In summary, given a tame $T$-module $A$, any extension of $A$ by $T$ is always a factor group of $G_\infty$.  $G_\infty$ is finitely presented and the defining relations are given by (\ref{commutative1})-(\ref{action2}) for any fixed positive real number $r\geqslant R$.

\section{Main Theorem}
\label{maintheorem}

\subsection{Preparation}
\label{maintheorem1}

Given a finitely presented metabelian group $G$ with the short exact sequence
\[1\to A \hookrightarrow G \twoheadrightarrow T \to 1, \]
such that $A, T$ are abelian and the torsion-free rank of $T$ is minimized. Since $G$ is finitely presented, in particular, it is finitely generated. Then $T$ is a finitely generated abelian group, and $A$ is finitely generated as a $T$-module (See \cite{Hall1954}). 

We can do the same trick as in \cref{finitepresented} to reduce the problem to a simpler case. Denote $\pi: G \twoheadrightarrow T$ to be the epimorphism such that $A\cong \ker \pi$. Let $T_1\leqslant T$ be the complement of the torsion subgroup of $T$. $G_1=\pi^{-1}(T_1)$ has finite index in $G$ then $G_1$ is quasi-isometric to $G$. It follows that $\delta_G=\delta_{G_1}$ due to \cref{quasi}. Therefore an upper bound of $\delta_{G_1}$ is also an upper bound for $\delta_{G}$. 

In addition, if $G_1$ can be written as an extension of two abelian groups $A_2$ and $T_2$, where the torsion-free rank of $T_2$ is strictly less than $k$, consider the following commutative diagram

\[
\begin{tikzcd}    
1 \arrow[r] &G_1 \arrow[r,"i"]\arrow[d,"\pi"] & G \arrow[r,""]\arrow[d,"f"] & G/G_1 \arrow[r]\arrow[d,"i"] & 1 \\
1 \arrow[r] &T_2 \arrow[r,"i"] & T_2\times G/G_1 \arrow[r,""] & G/G_1 \arrow[r] & 1 \end{tikzcd},
\]
where $f(g,h):=(\pi(g),h),g\in G_1,h\in G/G_1$. By the snake lemma, there exists an exact sequence 
\[
\begin{tikzcd}    
\ker \pi=A_2 \arrow[r,""] & \ker f \arrow[r,""] & \ker i=1 \arrow[r,""] &\mathrm{coker} \pi=1  \arrow[r,""]  &  \mathrm{coker} f \arrow[r,""] &\mathrm{coker} i=1.
\end{tikzcd}
\]
It follows that $\ker f$ is abelian and $f$ is surjective. Then $G$ can be represented as an extension of $\ker f$ by $T_2\times G/G_1$ where the torsion-free rank is of $T_2\times G/G_1$ strictly less than $k$. This contradicts the minimality of $k$. Therefore $k$ is preserved when passing to $G_1$.

Thus from now on, we shall assume that $T$ is a free abelian group of rank $k$ and $G$ is an extension of a tame $T$-module $A$ by the free abelian group $T$. Also let us assume that $k>0$.

Let $\mathcal T=\{t_1,\dots,t_k\}\subset G$ such that $\{\pi(t_1),\dots,\pi(t_k)\}$ forms a basis for $T$ and $\mathcal A$ be a finite subset of $G$ such that it contains all commutators $a_{ij}=[t_i,t_j]$ for $1\leqslant i<j \leqslant k$ and generates the $T$-module $A$. Then $\mathcal A\cup \mathcal T$ is a finite generating set for the group $G$.

By \cref{fp}, since $G$ is finitely presented,  $A$ is a tame $T$-module. Then there is a finite subset $\Lambda\subset C(A)\cup C(A^*)$ such that for each character $\chi: T\to \mathbb R$, there exists $\lambda\in \Lambda$ such that $\chi(\lambda)>0$. Let $F$ be the free group generated by $\mathcal T$ and $\bar F$ be the set of all ordered words in $F$ (See \cref{finitepresented}). Same as previous section, we let $\theta: F\to \mathbb R^k$ be the homomorphism given by 
\[\theta(t_i)=(\delta_{i1},\dots.\delta_{ik}), 1\leqslant i\leqslant k.\]
If $w\in F$ we shall write $\bar w$ for the unique word in $\bar F$ representing $w$ module $F'$. In addition, we denote $\tilde w\in T$ for the image of $w\in F$ under $\pi$.

Then we are able to define a sequence of groups $G_r$ as what we did in \cref{finitepresented}. But for our purpose, we will need a larger $R$. Let 
\[R=2k\max\{D^2/2C,D,D^2/(4kC-4)\},\]
where $C,D$ are defined in the same way in \cref{finiteness}. Since $R>2k\max\{D,D^2/C\}$, $G_R\cong G_\infty$, and in particular, $G$ is a factor group of the finitely presented group $G_\infty$. Then we can list all defining relations of $G_\infty$ here:

\begin{align}
	[t_i,t_j]&=a_{ij}, &\text{for } 1\leqslant i<j\leqslant k,\label{commutative1'}\\
	[a,b^u]&=1, &\text{for $a,b\in \mathcal A,u\in \bar F$ with $\|\theta(u)\|<R$},\label{commutative2'}\\
	\prod_{u\in \bar F} (a^{\lambda(\hat u)})^u&=a, &\text{for } a\in \mathcal A, \lambda\in \Lambda\cap C(A), \label{action1'}\\
	\prod_{u\in \bar F} (a^{\lambda(\hat u)})^{u^{-1}}&=a, &\text{for } a\in \mathcal A, \lambda\in \Lambda\cap C(A^*). \label{action2'}
\end{align}

 To simplify our notation, we will write relations (\ref{commutative1'}) and (\ref{commutative2'}) as $\mathcal R_1$ and relations (\ref{action1'}) and (\ref{action2'}) as $\mathcal R_2$.

Denote the epimorphism $\varphi: G_\infty \to G$ induced by the identity map on $\mathcal A\cup \mathcal T$. Note that $\varphi$ induces an isomorphism on $G_{\infty}/A_{\infty}\cong T$. Therefore $\ker \varphi\leqslant A_\infty$ is abelian where $A_\infty:=\llangle \mathcal A \rrangle_{G_\infty}\vartriangleleft G_\infty$. Let $\ker \varphi=\llangle \mathcal R_3\rrangle_{G_\infty}$, where $\mathcal R_3$ is a finite set. 

Thus we obtain a finite presentation for $G$. 
\begin{equation}
\label{presentation1}
	G=\langle \mathcal A\cup \mathcal T\mid \mathcal R_1\cup \mathcal R_2\cup \mathcal  R_3\rangle.
\end{equation} 

Then we have the following proposition:
\begin{proposition}
		\label{main2}
	Let $k$ be the torsion-free rank of $T$ such that $k>0$, then $\delta_G(n)\preccurlyeq 2^{n^{2k}}$.
\end{proposition}

\cref{main2} will be proved in \cref{maintheorem5}. 

\begin{proof}[Proof of \cref{main}]
	If $k=0$, $G$ has a finite index abelian subgroup. Therefore $\delta_G \preccurlyeq n^2$ by \cref{quasi}. 
	
	If $k>0$, the result follows directly from \cref{main2}.
\end{proof}

\subsection{The Ordered Form of Elements}
\label{maintheorem2}

For convenience, we assume that $|\mathcal A|=m$ and denote $\mathcal A=\{a_1,\dots,a_m\}$. For $g\in\llangle \mathcal A\rrangle_{H_\infty}$, we have  $\pi(g)=0$. Thus the sum of exponents of each $t_i$ is 0. Note that we only consider words that are fully reduced in $F(\mathcal A\cup \mathcal T)$, the free group generated by $\mathcal A\cup \mathcal  T$. Next, we introduce the \emph{ordered form} of elements in $\llangle \mathcal A\rrangle_{H_{\infty}}$, which helps us understand the module structure on the normal closure of $\mathcal A$.

We first claim that there exists an algorithm that allows us to write a word $w\in \llangle \mathcal A\rrangle_{H_\infty}$ in the form of a product of conjugates of elements in $\mathcal A$ in a unique way. 

Let us start with a word $w=u_1b_1u_2b_2\dots u_sb_su_{s+1}\in \llangle \mathcal A\rrangle_{H_\infty}$ where $u_i\in F,b_i\in \mathcal A^{\pm 1}$. Here $u_1,u_{s+1}$ could be empty. Then 
\[w=b_1^{u_1^{-1}}b_2^{(u_1u_2)^{-1}}\dots b_s^{(u_1u_2\dots u_s)^{-1}}u_1u_2\dots u_{s+1}.\]
The equality holds in the free group generated by $\mathcal A\cup \mathcal T$. Note that $u_1\dots u_{s+1}$ is a word in $F$. It also has the property that the sum of exponents of each $t_i$ is 0. We then write this word in the product of conjugates of $\{[t_i,t_j]^{\pm 1}, i<j\}$ algorithmically in the following fashion: assume we already write $u_1\dots u_{s+1}$ as $w_1w_2$ where $w_1$ is a product of conjugates of $\{[t_i,t_j]^{\pm 1}, i<j\}$ and $w_2$ is a word in $F$ such that the sum of exponents of each $t_i$ is 0. Let $t_i$ be the letter with the smallest indices among all letters in $w_2$. Then $w_2$ can be written as $w_2't_i^{\varepsilon}w_2'', \varepsilon=\pm 1$ where $w_2'$ does not contain any $t_i^{\pm 1}$. Then 
\[w_2't_i^{\varepsilon}w_2''=[t_i^{\varepsilon},t_{j_1}^{\varepsilon_1}]^{(w_2't_{j_1}^{-\varepsilon_1})^{-1}}[t_i^{\varepsilon},t_{j_2}^{\varepsilon_2}]^{(w_2't_{j_1}^{-\varepsilon_1}t_{j_2}^{-\varepsilon_2})^{-1}}\dots [t_i^\varepsilon,t_{j_l}^{\varepsilon_l}]t_i^{\varepsilon}w_2'w_2'',\]
where $w_2'=t_{j_l}^{\varepsilon_l}\dots t_{j_1}^{\varepsilon_1}.$ Since the sum of exponent of $t_i$ is 0, by repeating this process we can gather all $t_i$ to the left and hence they will be canceled eventually. We end up with a word $w_3w_4$ where $w_3$ is a product of conjugates of $\{[t_i,t_j]^{\pm 1}, i<j\}$ and $w_4$ is a word in $F$ such that the sum of exponents of each $t_i$ is 0 and of the length strictly less than $w_2$. Thus by repeating this algorithm, we are able to write $g$ as a product of conjugates of $\{[t_i,t_j]^{\pm 1}, i<j\}$ in a unique way. Now we just apply relations in (\ref{commutative1'}) replacing all the commutators by their corresponding letters in $\mathcal A$. Therefore the claim is proved. 

Since $g$ can be written as a product of conjugates of elements in $\mathcal A$, applying commutator relations like $[a,b^u], a,b\in \mathcal A,u\in F$, we are able to commute those conjugates hence gather all conjugates which share the same base. In addition, combining the fact $a^u=a^{\bar u}$ from \cref{HP} (a), we can write $g$ in a \emph{ordered form} of the following type
\[g=a_1^{\lambda_1}\dots a_m^{\lambda_m},\]
where $\lambda_i\in \mathbb ZT$ and terms of $\lambda_i$ are written in the order from the high to low with respect to $\prec$ which we define in \cref{membership}. For $w$ in $\llangle \mathcal A\rrangle_{G_\infty}$ (or $\llangle \mathcal A\rrangle_G$), we define the ordered form as the ordered form of $\iota(w)$ where $\iota: G_{\infty}\to H_\infty$ (resp. $G\to H_\infty$) is the combinatorial map induced by identity on $\mathcal A\cup \mathcal T$. Note that by the way we define the ordered form, the ordered form of each word is unique. The ordered forms distinguish different elements in the $T$-module $\llangle \mathcal A\rrangle_{H_\infty}$. In fact, two elements in $\llangle \mathcal A\rrangle_{H_\infty}$ are equal in $H_\infty$ if and only if they have the same ordered form. One remark is that two words which are equal in $G$ or $G_{\infty}$ may have different ordered forms, for example, $a_1$ and $\prod_{u\in \bar F} (a^{\lambda(\hat u)})^u$, $\lambda\in C(A)$.

Recall that $G=\langle \mathcal A\cup \mathcal T\mid \mathcal R_1\cup \mathcal R_2\cup \mathcal  R_3\rangle$. Note that both $\mathcal R_2$ and $\mathcal R_3$ are contained in the normal closure of $\mathcal A$. From now on we write relators from $\mathcal R_2\cup \mathcal R_3$ in their ordered form. 

\subsection{Main Lemmas}
\label{maintheorem3}

Before we embark on the proof of \cref{main2}, we shall establish some preliminary lemmas.

Now consider an arbitrary factor group $H$ of $G_\infty$ equipped with the presentation
\begin{equation}
\label{presentation03}
	H=\langle \mathcal A \cup \mathcal T\mid \mathcal R_1\cup \mathcal R_2\cup \mathcal R\rangle
\end{equation}
where $\mathcal R$ is a finite subset in $G_\infty$. Then $H\cong G_\infty/\llangle \mathcal R \rrangle_{G_\infty}$. Note that if $\mathcal R=\mathcal R_3$, $H=G$, if $\mathcal R=\emptyset$, $H=G_\infty$ which are two major examples we concern. We have the following lemmas for $H$.
\begin{lemma}
	\label{abelian}
	Let $H$ be a factor group of $G_\infty$ equipped with presentation (\ref{presentation03}) and $w$ be a word in $(\mathcal T\cup \mathcal T^{-1})^*$ such that $|w|=n$, then 
	\[w=_H \bar w\prod_{i=1}^p b_i^{u_i}\]
	where $p\leqslant n^2, b_i\in \mathcal A^{\pm 1}, u_i \in F, \Tr(\theta(u_i))\subset B_n$. In addition, the number of relations cost to convert LHS to RHS is bounded by $n^2$.
\end{lemma}

\begin{proof}
	Since $\bar w=t_1^{m_i}\dots t_k^{m_k}$ for some $m_1,\dots,m_k\in \mathbb Z$ such that $\sum_{i=1}^k |m_i|\leqslant n$, to move each letter in $w$ to the desired place, it will cost at most $n$ commutators of the form $[t_i,t_j], 1\leqslant i<j \leqslant k$. In total, we need at most $n^2$ such commutators. That is,
	\[w= \bar w\prod_{i=1}^p [t_{i_1},t_{i_2}]^{\varepsilon_i u_i'}, \text{ where } u_i'\in F, p\leqslant n^2, 1\leqslant i_1<i_2\leqslant k, \varepsilon_i\in\{\pm 1\}.\]
	Moreover, since the length of $w$ is bounded by $n$, $\Tr(\theta(u_i'))\leqslant n$.
	
	By applying relations in $\{a_{ij}=[a_i,a_j]\mid 1\leqslant i<j\leqslant n\}$ $p$ times we immediately have
	\[w=_H \bar w\prod_{i=1}^p b_i^{u_i}, \Tr(\theta(u_i))\leqslant n.\]
	The cost of relations is bounded by $p\leqslant n^2$. 
\end{proof}

In particular, for $w\in F$ such that $\pi(w)=1$, it costs at most $n^2$ relations in $H$ to convert it to a product of conjugates of elements in $\mathcal A$. 

\begin{lemma}
\label{conjugate}
Let $H$ be a factor group of $G_\infty$ equipped with presentation (\ref{presentation03}) then there exists a constant $K$ only depends on $\mathcal R_1\cup \mathcal R_2$ such that 
\[\area([a,b^u])\leqslant K^n,\forall a,b\in \mathcal A, \|\theta(u)\|<n.\]
\end{lemma}

\begin{proof}
	Let $\mathcal F=\{\theta(\supp(\lambda))\mid \lambda \in \Lambda\}$ then $\mathcal F$ is a finite colletction of finite sets. By the choice of $\Lambda$, $\mathcal F$ satisfies the assumptions of \cref{geometric}. 
	
	By \cref{geometric}, each $x\in B_{r+\varepsilon(r)}$ can be taken from $B_r$ by $\mathcal F$ for $r>R$. Recall that $R$ is defined to be $\max\{D,D^2/2C,D^2/(4kC-4)\}$ and $\varepsilon(r)=C-D^2/2r$ , where $C,D$ are purely determined by $\Lambda$ hence $\mathcal R_1\cup \mathcal R_2$ as we stated in \cref{finiteness}.
	
According to our choice of $R$, we note that $\varepsilon(r)\geqslant \varepsilon(2kD^2/(4kC-4))=\frac{1}{2k}$ for $r>R$. Let $K_1$ be the constant which is large enough such that $f(n)\leqslant K_1^n$ for $n\leqslant R$, and $K_2$ be the constant  
	\[K_2:=\max_{\lambda\in \Lambda}\{\sum_{u\in \bar F} |\lambda(u)|\}+2.\]
	Since each $\lambda$ has finite support, $K_2$ is well-defined. Now let $K:=\max\{K_1,K_2^{2k}\}.$
	
	Suppose for $n>R$, $\area([a,b^u])\leqslant K^n, \forall a,b\in \mathcal A, \|\theta(u)\|< n$. We then prove our lemma by induction. Let us first consider the case $r=n+\frac{1}{2k}$. Fix some $v\in \bar F$ satisfying $\|\theta(v)\|< r$. Since $\varepsilon(n)\geqslant \frac{1}{2k}$, $B_{n+\frac{1}{2k}}$ can be taken from $B_n$ by $\mathcal F$. Then there is $\lambda\in \Lambda$ with $\theta(\supp(\lambda \hat v))\subset B_n$ by the definition of ``taken from''. 
	
	Therefore we have two cases depending on $\lambda\in C(A)$ or $C(A^*)$. Firstly assume that $\lambda\in \Lambda\cap C(A)$. Then by applying the commutator formula $[x,yz]=[x,y]^{x^{-1}zx}[x,z]$, we obtain
	\[[a,b^v]=_G[a,\prod_{u\in \bar F} (b^{\lambda(\hat u)})^{uv} ]=\prod_{u\in \bar F}[a,b^{\lambda(\hat u)uv}]^{h(u)},\]
	where the $h(u)$'s are certain elements in $H$ which need not concern us. Note that in the first equality above, we apply relations in (\ref{action1'}) twice to replace $b$ by $\prod_{u\in \bar F} b^{\lambda(y)u}$. Since $\supp(\lambda)\subset \bar B_D$ we have $\|\theta(u)\|<D<\frac{n}{2k}.$ Additionally, we have $\|\theta(v)\|< n+\frac{1}{2k}$ and $\|\theta(uv)\|<n$. It meets all assumptions of \cref{HP} (c). Note that $H$ is a factor group of $H_n$ which we defined in \cref{finitepresented}. Then $[a,b^{\lambda(\hat u)uv}]$ is conjugate in $H$ to $[a,b^{\lambda(\hat u)\overline{uv}}]$, the area of which is bounded by $|\lambda(\bar u)|K^n$. It follows that
	\[\area([a,b^v])\leqslant 2+\sum_{u\in \bar F} \area([a,b^{\lambda(\hat u)uv}])\leqslant 2+\sum_{u\in \bar F} |\lambda(\hat u)|K^n\leqslant K_2 K^n.\]
	Repeating this process $2k$ times, we obtain that 
	\[\area([a,b^v])\leqslant K_2^{2k} K^n\leqslant K^{n+1}, \text{ for } v\in \bar F, \|\theta(v)\|<n+1.\]  
	If $\lambda\in \Lambda\cap C(A^*)$, the only different is that 
	\[[a,b^v]=[a^{v^{-1}},b]^{v^{-1}}=[\prod_{u\in \bar F}(a^{\lambda(\hat u)})^{uv^{-1}},b]^{v^{-1}}.\]
	Similarly we obtained that 
	\[\area([a,b^v])\leqslant K_2^{2k}K^n\leqslant K^{n+1}.\]
\end{proof}

Furthermore, \cref{conjugate} allows us to estimate the cost to commute two conjugates of elements in $\mathcal A$. Since the normal closure of $\mathcal A$ in $H$ is abelian, this lemma provides a tool to estimate the cost for converting words in $\llangle \mathcal A\rrangle_H$, in particular, in $G$. Also \cref{conjugate} reveals how much metabelianness cost in a finitely presented metabelian group. We will discuss this in \cref{relativedehn1}.
\begin{lemma}
\label{organizer}
Let $H$ be a factor group of $G_\infty$ equipped with presentation (\ref{presentation03}) and $K$ be the same constant in \cref{conjugate}. Then in $H$ we have
\[\area(a^ua^{-\bar u})\leqslant (2K)^n,\forall a\in \mathcal A,u\in F, \Tr(u)\subset B_n\]
\end{lemma}

\begin{proof}
	We prove it by an induction on $n$. Suppose for $i\leqslant n$, the result holds. Then for the case $n+1$, we write $u=u't_s^{\pm 1}$ then $\Tr(u)\subset B_{n+1},\Tr(u')\subset B_n$. 
	\[a^u=a^{u't_s^{\pm 1}}=(a^{\bar u'})^{t_s^{\pm 1}}\nu_1,\]
	where $\area(\nu_1)\leqslant (2K)^n$ by our inductive assumption. Write $\bar u'=t_1^{m_1}\dots t_k^{m_k}$, we claim that 
	\[\bar u't_s^{\pm 1}=\overline{u}\prod_{j=1}^m c_j^{\alpha_j} \text{ where } c_j\in \{[t_s,t_l]^{\pm 1}\mid 1\leqslant s<l\leqslant k\}, \alpha_j\in \bar F, m=\sum_{i=s+1}^k |m_i|\leqslant n. \]
	We need to be really careful here. Let us first consider the case that the exponent of $t_s$ is 1. We assume $s<k$, otherwise it is trivial. Note that if $m_k\geqslant 0$
	\begin{align*}
			t_k^{m_k}t_s=t_k^{m_k-1}t_st_k[t_s,t_k]^{-1}&=t_k^{m_k-2}t_st_k^2[t_s,t_k]^{-t_k}[t_s,t_k]^{-1}\\
			&=t_k^{m_k-3}t_st_k^3[t_s,t_k]^{-t_k^2}[t_s,t_k]^{-t_k}[t_s,t_k]^{-1} \\
			&\vdots \\
			&=t_st_k^{m_k}[t_s,t_k]^{-t_k^{m_k-1}}\dots [t_s,t_k]^{-t_k}[t_s,t_k]^{-1}.
	\end{align*}
    If $m_k<0$, we have
     \begin{align*}
			t_k^{m_k}t_s=t_k^{m_k+1}t_st_k^{-1}[t_s,t_k]^{t_s}&=t_k^{m_k+2}t_st_k^{-2}[t_s,t_k]^{t_st_k^{-1}}[t_s,t_k]^{t_s}\\
			&=t_k^{m_k+3}t_st_k^{-3}[t_s,t_k]^{t_st_k^{-2}}[t_s,t_k]^{t_st_k^{-1}}[t_s,t_k]^{t_s^{-1}} \\
			&\vdots \\
			&=t_st_k^{m_k}[t_s,t_k]^{t_st_k^{m_k+1}}\dots [t_s,t_k]^{t_st_k^{-1}}[t_s,t_k]^{t_s}.
	\end{align*}
	Repeating this process, we then prove the claim for the case that the exponent of $t_s$ is 1.
	
	On the other hand, if the exponent of $t_s$ is $-1$, then similarly, consider if $m_k\geqslant 0$
	 \begin{align*}
			t_k^{m_k}t_s^{-1}=t_k^{m_k-1}t_s^{-1}t_k[t_s,t_k]^{t_k}&=t_k^{m_k-2}t_s^{-1}t_k^{2}[t_s,t_k]^{t_k^2}[t_s,t_k]^{t_s}\\
			&=t_k^{m_k-3}t_s^{-1}t_k^{3}[t_s,t_k]^{t_k^{3}}[t_s,t_k]^{t_k^{2}}[t_s,t_k]^{t_k} \\
			&\vdots \\
			&=t_s^{-1}t_k^{m_k}[t_s,t_k]^{t_k^{m_k-1}}\dots [t_s,t_k]^{t_k^{2}}[t_s,t_k]^{t_k},
	\end{align*}
 and if $m_k<0$
 \begin{align*}
			t_k^{m_k}t_s^{-1}=t_k^{m_k+1}t_s^{-1}t_k^{-1}[t_s,t_k]^{-1}&=t_k^{m_k+2}t_s^{-1}t_k^{-2}[t_s,t_k]^{-t_k^{-1}}[t_s,t_k]^{-1}\\
			&=t_k^{m_k+3}t_s^{-1}t_k^{-3}[t_s,t_k]^{-t_k^{-2}}[t_s,t_k]^{-t_k^{-1}}[t_s,t_k]^{-1} \\
			&\vdots \\
			&=t_s^{-1}t_k^{m_k}[t_s,t_k]^{-t_k^{m_k+1}}\dots [t_s,t_k]^{-t_k^{-1}}[t_s,t_k]^{-1}.
	\end{align*}
	Again by repeating this process, the claim is proved. Thus by induction on $k$, we can move $t_s$ to the desired place.
    
    Now we have
    \[a^u=a^{\bar u't_s^{\pm 1}}\nu_1=(\prod_{j=1}^m c_j^{\alpha_j})^{-1}a^{\bar u}(\prod_{j=1}^m c_j^{\alpha_j})\nu_1.\]
    Apply relations from $\{a_{ij}=[a_i,a_j]\mid 1\leqslant i<j\leqslant n\}$ $2m$ times, we have that 
    \[a^u=(\prod_{j=1}^m d_j^{\alpha_j})^{-1}a^{\bar u}(\prod_{j=1}^m d_j^{\alpha_j})\nu_2\nu_1\]
    where $d_j\in \mathcal A^{\pm 1}$ and $\area(\nu_2)\leqslant 2m$ by our disccusion.
    
    Next we need to commute $a^{\bar u}$ and $d_j^{\alpha_j}$ for $j=1,\dots,m$ to the left and estimate the cost. Note that $[a^{\bar u},d_j^{\alpha_j}]$ is conjugate to $[a,d_j^{\alpha_j (\bar u)^{-1}}]$. From the computation above, $\alpha_j$ is either a tail of $\bar u$ or a tail of $\bar u$ multiplied by $t_s^{\pm 1}$. Therefore $(\bar u)^{-1}$, $\alpha_j$, $\alpha_j(\bar u)^{-1}$ satisfy the assumption of \cref{HP} (b). Thus $[a,d_j^{\alpha_j (\bar u)^{-1}}]$ is conjugate to $[a,d_j^{\overline{\alpha_ju^{-1}}}]$. Since $\|\theta(\overline{\alpha_ju^{-1}})\|\leqslant n+1$, the area of $[a,d_j^{\overline{\alpha_ju^{-1}}}]$, by \cref{conjugate}, is bounded by $K^{n+1}$.
    	
	Applying $[a,d_j^{\overline{\alpha_ju^{-1}}}]$ to $a^u$ and $d_j^{\alpha_j}$ for $j=1,\dots, m$, we can commute all $d_j^{\alpha_j}$ to the left such that it cancels with $d_{j}^{-\alpha_j}$. Then we finally have
	\[a^u=a^{\bar u}\nu_3\nu_2\nu_1,\]
	where 
	\[\area(\nu_3)\leqslant mK^{n+1}\]
	In total, the cost of converting $a^u$ to $a^{\bar u}$ is bounded by 
	\[\area(\nu_3\nu_2\nu_1)\leqslant \area(\nu_3)+\area(\nu_2)+\area(\nu_1)\leqslant (2K)^n+2m+mK^{n+1}\leqslant (2K)^{n+1}.\]
	Note that we use the fact that $m\leqslant n$ and we can choose $K\gg 1$.
\end{proof}

\cref{organizer} provides a method for us to ``organize'' the exponent of a conjugate. In particular, combining \cref{conjugate} and \cref{organizer}, we are able to convert any word in $\llangle \mathcal A\rrangle_{H}$ to its ordered form.

\subsection{The $T$-module in Metabelian Groups}
\label{maintheorem4}

In $H_\infty$, $\llangle \mathcal A \rrangle_{H_\infty}$ is naturally a $T$-module by the conjugation action. Let $\mathcal A=\{a_1,\dots,a_m\}$. For each element $g\in \llangle \mathcal A \rrangle_{H_\infty}$, it can be written in its ordered form, i.e.
\[g=\prod_{i=1}^m a_1^{\lambda_1}a_2^{\lambda_2}\dots a_m^{\lambda_m}\in \mathbb ZT.\]
For $\lambda_i$, we always write it from high to low with respect to the order $\prec$. Then $g$ can also be regarded as an element $(\lambda_1,\dots,\lambda_m)$ in the free $T$-module with basis $a_1,\dots,a_m$. From now on, we treat an element in $\llangle \mathcal A\rrangle_{H_\infty}$ as an element in group $H_\infty$ as well as an element in the free $T$-module with basis $a_1,\dots, a_m$.

Let us first state the relation of operations between the group language and module language:
\begin{center}
  \begin{tabular}{ l | l }
    \hline
    Group & Module \\ \hline
    $(\prod_{i=1}^m a_i^{\lambda_i})(\prod_{i=1}^m a_i^{\lambda_i'})=_{H_\infty}\prod_{i=1}^m a_i^{\lambda_i+\lambda_i'}$ & $(\lambda_1,\dots,\lambda_m)+(\lambda_1',\dots,\lambda_m')=(\lambda_1+\lambda_1',\dots,\lambda_m+\lambda_m')$ \\ \hline
    $(\prod_{i=1}^m a_i^{\lambda_i})^c=_{H_\infty}\prod_{i=1}^m a_i^{c\lambda_i}$ & $c(\lambda_1,\dots,\lambda_m)=(c\lambda_1,\dots,c\lambda_m)$ \\ \hline
    $(\prod_{i=1}^m a_i^{\lambda_i})^t=_{H_\infty}\prod_{i=1}^m a_i^{t\lambda_i}$ & $t(\lambda_1,\dots,\lambda_m)=(t\lambda_1,\dots,t\lambda_m)$  \\ 
    \hline
  \end{tabular}
\end{center}
where $c\in \mathbb C,t\in T$. 

Let $\mathcal X$ be a subset of $\llangle \mathcal A \rrangle_{H_\infty}$. Then the normal closure of $\mathcal X$ in group $H_\infty$ coincides the submodule generated by $\mathcal X$ over $\mathbb ZT$. One direction is trivial, since by the table we have above, elements lie in the submodule are obtained by the group operations and conjugations. Conversely, let $g\in \llangle \mathcal A \rrangle_{H_\infty}$ then if $h\in \mathbb ZT$, $g^h$ can be obtained by finitely many scalar products and module operations and if $h\in \llangle \mathcal A\rrangle_{H_\infty}$, then $g^h=g$. The general case is a combination of those two cases. Thus $g^h$ must lie in the submodule generated by $g$. On the contrary, the subgroup generated by $\mathcal X$ coincides with the submodule generated by $\mathcal X$ over $\mathbb Z$.

Again we consider an arbitrary factor group $H$ of $G_\infty$ with the finite presentation 
\[H=\langle \mathcal A \cup \mathcal T\mid \mathcal R_1\cup \mathcal R_2\cup \mathcal R\rangle\]
where $\mathcal R$ is a finite subset of $G_\infty$. Then $H\cong G_\infty/\llangle \mathcal R \rrangle$. We now estimate the cost of relations in group $H$ to make each of the module operations above. Note that notations like $\deg (\lambda)$ and $|\lambda|$ for element $g\in \mathbb ZT$ are inherited from the polynomial ring $\mathbb Z[t_1,t_1^{-1}\dots,t_k,t_k^{-1}]$ (See in \cref{membership}).

In the following lemma, $K$ is the same constant appeared in \cref{conjugate}, which only depends on $\mathcal R_1\cup \mathcal R_2$.

\begin{lemma}
\label{thelemma}
Let $H$ be a factor group of $G_\infty$ equipped with presentation (\ref{presentation03}) then we have
\begin{enumerate}[(a)]
	\item Let 
	\[f=\prod_{i=1}^m a_i^{\lambda_i},g=\prod_{i=1}^m a_i^{\lambda_i'},\]
	and we denote $P=\max\{|\lambda_i|,|\lambda_i'|\mid i=1,\dots,m\}$, $Q=\max\{\deg(\lambda_i),\deg(\lambda_i')\mid i=1,\dots,m\}$. Then the cost of relations in $H$ of converting 
	\[fg=_{H}\prod_{i=1}^m a_i^{\lambda_i+\lambda_i'}\]
	is at most $m^2P^2 K^{2Q}$ where the right hand side is written in its ordered form.
	\item Let 
    \[f=\prod_{i=1}^m a_i^{\lambda_i},\]
    denote $P=\max\{|\lambda_i|\mid i=1,\dots,m\}$, $Q=\max\{\deg(\lambda_i)\mid i=1,\dots,m\}$. For $c\in \mathbb Z$, the cost of relations in $H$ of converting 
    $f^c$ to $\prod_{i=1}^m a_i^{c\lambda_i}$	
    is at most $(|c|-1)(m)^{2}P^{2}K^{2Q}$ where the right hand side is written in its ordered form.
	\item Let 
    \[f=\prod_{i=1}^m a_i^{\lambda_i},\]
    denote $P=\max\{|\lambda_i|\mid i=1,\dots,m\}$, $Q=\max\{\deg(\lambda_i)\mid i=1,\dots,m\}$. For $t\in T$, the cost of relations in $H$ of converting
    \[(\prod_{i=1}^m a_i^{\lambda_i})^t=_{H}\prod_{i=1}^m a_i^{t\lambda_i}, \]
    is bounded by $(mP)(2K)^{k(Q+\deg t)}$.
\end{enumerate}
\end{lemma}

\begin{proof}
	\begin{enumerate}[(a)]
		\item First we consider a simpler case when $g=a_1^{\lambda_1'}$. Then it is essential to estimate the cost of converting LHS to RHS of 
		\begin{equation}
		\label{eqlemma}
			(\prod_{i=1}^m a_i^{\lambda_i})a_1^{\lambda_1'}=_H(a_1^{\mu_1})(a_2^{\lambda_2}\dots a_m^{\lambda_m}), \mu_1=\lambda_1+\lambda'_1.
		\end{equation}
		In order to commute $a_1^{\lambda_1'}$ with $a_m^{\lambda_m},\dots,a_2^{\lambda_2}$, we apply \cref{conjugate} $(m-1)$-times. Each step costs at most $PK^{2Q}$ since $\deg(\lambda_i+\lambda_1')\leqslant 2Q,|\lambda_i|,|\lambda_1'|\leqslant P$. Therefore, the cost of 
		\[(\prod_{i=1}^m a_i^{\lambda_i})a_1^{\lambda_1'}=(a_1^{\lambda_1}a_1^{\lambda_1'})(a_2^{\lambda_2}\dots a_m^{\lambda_m})\]
		is bounded by $(m-1)P^2K^{2Q}$. When it comes to the last step, i.e
		\[a_1^{\lambda_1}a_1^{\lambda_1'}=a_1^{\mu_1}.\]
		the only thing we need to do is move each term $a_1^u u\in F$ to its position corresponding to $\prec$. We in fact sort all conjugates $a_1^u$ in order. Note that those conjugates in $a_1^{\lambda_1}$ and $a_1^{\lambda_1'}$ are already in order, respectively. Thus we only need to insert each $a_1^u$ of $a_1^{\lambda_1'}$ into terms of $a_1^{\lambda_1}$. Again from \cref{conjugate}, the cost is bounded by $P^2K^{2Q}$. 
		
		Therefore, the cost of (\ref{eqlemma}) is bounded by $mP^2K^{2Q}$.
		
		In general, if $g=\prod_{i=1}^m a_i^{\lambda_i'}$. By repeating previous process $m$ times, we get an upper bound $m^2P^2K^{2Q}$. We complete the proof.
		
		\item It follows by applying (a) $|c|-1$ times.
	    \item Conjugating $t$ to each term of $a_1^{t'}$, $t'\in T$, cost zero relations. Then we basically estimate the cost of the following equatioin
	    \[a_1^{t't}=a_1^{\overline{t't}}.\]
	    By the result of \cref{organizer}, since $\Tr(t't)\subset B_{(\deg t+\deg t')}$, then the cost is bounded by $(2K)^{(\deg t+\deg t')}.$ Notice $\deg(t')\leqslant Q$, then the total cost is at most
	    \[(mP)(2K)^{(\deg t+\deg t')}\leqslant (mP)(2K)^{(Q+\deg t)}. \]
	    Here we use the fact $\Tr(t)\subset B_{\deg (t)}$ since we order elements in $\mathbb ZT$ degree lexicographically.
	\end{enumerate}
\end{proof}

Recall that $G_\infty$ is a factor group of $H_\infty$ as $G$ is a factor group of $G_\infty$. Denote the epimorphism from $H_\infty$ to $G_\infty$ induced by identity on generating set as $\psi$, and then we have the following homomorphism chain:
\[\begin{tikzcd}
    H_\infty \arrow[r,"\psi"] & G_\infty \arrow[r,"\varphi"] & G.
\end{tikzcd}\]
Thus $\ker \psi=\llangle \mathcal R_2\rrangle_{H_\infty}, \ker(\varphi\circ \psi)=\llangle \mathcal R_2\cup \mathcal R_3\rrangle_{H_\infty}.$ They are all normal subgroups in $H_\infty$ as well as submodules in $\llangle\mathcal A\rrangle_{H_\infty}$. $H_\infty$ contains a free module structure while each of $G$ and $G_\infty$ contain a factor module of it. Eventually we will convert the word problem to a membership problem of a submodule in $\llangle \mathcal A \rrangle_{H_\infty}$.
\subsection{Proof of \cref{main2}}
\label{maintheorem5}

Now we are ready to prove \cref{main2}. It is enough to show that for any given word $w=1$ of length $n$, $w$ can be written as a product of at most $C^{n^{2k}}$ conjugates of relators for some constant $C$. Since $G$ is a factor group of $H_\infty$, $w=_G 1$ if and only if $w\in \ker(\varphi \circ \psi)=\llangle \mathcal R_2\cup \mathcal R_3 \rrangle_{H_\infty}$. Note that $\llangle \mathcal R_2\cup  \mathcal R_3\rrangle_{H_\infty}\subset  \llangle \mathcal A \rrangle_{H_\infty}$. Recall that $\llangle\mathcal A\rrangle_{H_\infty}$ has a natural module structure: it is a free $T$-module with basis $a_1,\dots, a_m$. By previous discussion, $\llangle \mathcal R_2\cup \mathcal R_3\rrangle_{H_\infty}$ coincides the submodule generated by $\mathcal R_2\cup \mathcal R_3$ over $\mathbb ZT$. Let $\mathcal R_4=\{f_1,f_2,\dots,f_l\}$ be the Gr\"{o}bner basis of the submodule generated by $\mathcal R_2\cup \mathcal R_3$. We then add $\mathcal R_4$ to our presentation (\ref{presentation1}), and in addition we assume that all relators of $\mathcal R_i, i=2,3,4$ are written in their ordered form. Note that $\mathcal R_2\cup \mathcal R_3$ and $\mathcal R_4$ generates the same submodule in $\llangle \mathcal A\rrangle_{H_\infty}$. It implies that $\llangle \mathcal R_2\cup \mathcal R_3\rrangle_{H_\infty}=\llangle \mathcal R_4\rrangle_{H_\infty}$. We obtained an alternating presentation of $G$ as 
	\begin{equation}
	\label{newpresentation}
		G=\langle \mathcal A\cup \mathcal T\mid \mathcal R_1\cup \mathcal R_2\cup \mathcal R_3\cup \mathcal R_4\rangle.
	\end{equation}
	Although $\mathcal R_4$ is equivalent to $\mathcal R_2\cup \mathcal R_3$, it is convenient to keep $\mathcal R_2$, $\mathcal R_3$ in our presentation since all the estimation we have done previously are based on $\mathcal R_2\cup\mathcal R_3$.
	
    Note that $G$ is a factor group of $G_\infty$ and with the given presentation \cref{abelian}, \cref{conjugate}, \cref{organizer} and \cref{thelemma} all hold for $G$. 
    
    Since Dehn function is a quasi-isometric invariant then it enough for us to prove \cref{main2} using the presentation (\ref{newpresentation}).
    
\begin{proof}[Proof of \cref{main2}]
We start with a word $w\in G$ such that $|w|=n, w=_G 1$. WLOG we may assume 
\[w=u_1b_1u_2b_2\dots u_sb_su_{s+1}\]
where $u_i\in F=F(\mathcal T), b_i\in \mathcal A^{\pm 1}$ and $s+\sum_{i=1}^{s+1}|u_i|=n$. Let $v_i=(u_1\dots u_i)^{-1}$ for $i=1,\dots, s$ and $\nu=u_1 u_2\dots u_{s+1}$. Then we have 
\[w=w_1:=b_1^{v_1}b_2^{v_2}\dots b_s^{v_s}\nu.\]
The equality holds in the free group generated by $\mathcal A\cup \mathcal T$ thus the cost of relations converting $w$ to $w_1$ is 0. Since $s+\sum_{i=1}^{s+1}|u_i|=n$, in particular, we have that $s\leqslant n$. Moreover $|v_i|=\sum_{j=1}^i |u_j|\leqslant n$ hence $\Tr(\theta(v_i))\subset B_n, i=1,2,\dots,n$. 
    Next since $w_1=_G 1$, $\pi(w_1)=\pi(v_{s+1})=1 $. By \cref{abelian}, 
\[\nu=\prod_{i=s+1}^{s'} b_i^{v_i}\]
where $s'-s\leqslant |\nu|^2\leqslant n^2$, $b_i\in \mathcal A^{\pm 1}$, and $\Tr(\theta(v_i))\subset B_n, i=s+1,\dots, s'$. By \cref{abelian}, the cost of converting $\nu$ to the right hand side is bounded by $|\nu|^2\leqslant n^2$. 

Thus we let 
\[w_2:=\prod_{i=1}^{s'} b_i^{v_i}, s'\leqslant n^2+n, \Tr(\theta(v_i))\subset B_n, i=1,\dots, s'.\]
And the cost from $w_2$ to $w_1$ is bounded by $n^2$.

Next, note that all $v_i$'s are words in $F$. With the help of \cref{organizer}, we are able to organize $v_i$ to its image in $\bar F$. More precisely, we let
\[w_3:=\prod_{i=1}^{s'} b_i^{\bar v_i}, s'\leqslant n^2+n, \|\theta(\bar v_i)\|\leqslant n.\]
Also followed by \cref{organizer}, $w_2=_G w_3$. Let us estimate the cost of converting $w_2$ to $w_3$. To transform $w_2$ to $w_3$, we need apply \cref{organizer} to each $b_i^{v_i}$ once. Since $\area(b_i^{v_i}b_i^{-\bar v_i})\leqslant (2K)^n$ which provided by $\Tr(\theta(v_i))\subset B_n$, each transformation costs $(2K)^n$ relations. We have in total $s'\leqslant n+n^2$ many conjugates to convert therefore the cost is bounded by $(n^2+n)(2K)^n$.

Now let $w_4$ be the ordered form of $w_3$, which in fact is also the ordered form of $w$, i.e.
\[w_3=_G w_4:=\prod_{i=1}^m a_i^{\mu_i}\]
where $\mu_i$ are ordered under $\prec$. By the discussion in \cref{maintheorem2}, we obtain the ordered form just by rearranging all conjugates of $\mathcal A^{\pm 1}$. Note that because $\|\theta(\bar v_i)\|\leqslant n$ for all $i$, it cost at most $K^{2n}$ relations to commute any two consecutive conjugates $b_i^{\bar v_i}$ and $b_j^{\bar v_j}$ by \cref{conjugate}. To sort $s'$ conjugates we need commute $s'^2$ times. Therefore the number of relations need to commute $w_3$ to $w_4$ is bounded above that $s'^2K^{2n}\leqslant (n^2+n)^2K^{2n}$.

The only thing remains is to compute the area of $w_4$. Recall that $w_4$ can be regarded as an element in a free $T$-module generated by $a_1,\dots,a_m$. $w_4=_G 1$ implies that either $w_4=_{H_\infty} 1$ or it lies in the submodule generated by $\mathcal R_4=\{f_1,f_2,\dots,f_l\}$ which is the Gr\"{o}bner basis of the submodule generated by $\mathcal R_2\cup \mathcal R_3$. If $w=_{H_\infty} 1$ then $\mu_i=\emptyset$ for all $i=1,\dots, m$. In this case $\area(w_4)=0$. Thus 
\[\area(w)\leqslant n^2+(n^2+n)(2K)^n+(n^2+n)^2K^{2n}.\]
We are done with this case.

Now let us consider the case $w\in \llangle \mathcal R_4\rrangle_{H_\infty}\setminus \{1\}$. Let $K$ be a constant large enough to satisfy both \cref{division} and \cref{conjugate}. As an element in the $T$-module, $\deg w_4\leqslant n$ since $\|\theta(\bar v_i)\|\leqslant n$ for all $i$. Also recall that for an element $\alpha\in \mathbb ZT$, $|\alpha|$ is defined to be the $l_1$-norm  of it regarded as a finite suppported function from $T$ to $\mathbb Z$. Thus $|w_4|$ represents the number of conjugates in $w_4$ which is $s'$. Then by \cref{division} we have 
\[w_4=_{H_\infty} \prod_{i=1}^l f_i^{\alpha_i}, f_i=a_1^{\mu_{i1}}a_2^{\mu_{i2}}\dots a_m^{\mu_{im}}\in \mathcal R_4, \deg(f_i^{\alpha_i})\leqslant n, \sum_{i=1}^l |\alpha_i|\leqslant s'K^{n^{2k}}\leqslant (n^2+n)K^{n^{2k}}.\]
where $\mu_i=\sum_{j=1}^l \alpha_j\mu_{ji}$ in $\mathbb ZT$. Note that $f_i^{\alpha_i}$ is the product consisting of exactly $|\alpha_i|$ many relators. In conclusion we have 
\[\area(\prod_{i=1}^l f_i^{\alpha_i})\leqslant \sum_{i=1}^l |\alpha_i|\leqslant (n^2+n)K^{n^{2k}}.\]
Last, let us estimate the cost of converting $\prod_{i=1}^l f_i^{\alpha_i}$ to $w_4$. This process consists of two different steps: 1. converting all $f_i^{\alpha_i}$'s to their ordered form; 2. adding the $l$ terms of ordered $f_i^{\alpha_i}$.

To transform $f_i^{\alpha_i}$ to its ordered form, we write 
\[\alpha_i=\sum_{u\in \supp \alpha_i} \alpha_i(u)u.\]
Let us denote $P=\max_{i=1}^l |f_i|, Q=\max_{i=1}^l \deg(f_i)$.
Then 
\begin{align}
\label{w4}
	f_i^{\alpha_i}&=f_i^{\sum_{u\in \supp{\alpha_i}}\alpha_i(u)u}
=\prod_{\supp \alpha_i} f_i^{\alpha_i(u)u}
=\prod_{\supp \alpha_i} f_{i,u} 
=a_1^{\mu_{i1}'}a_2^{\mu_{i2}'}\dots a_m^{\mu_{im}'}=:f_i',
\end{align}
where $f_{i,u}$ is the ordered form of $f_i^{\alpha_i(u)u}$ and $u_{ij}'=\alpha_i\mu_{ij}$ hence $f_i'$ is the ordered form of $f_i^{\alpha_i}$. The first two equalities above hold in the free group $F(\mathcal A\cup \mathcal T)$ thus the cost is 0. In the third equality, applying \cref{thelemma} (b) and (c), the cost of converting $f_i^{\alpha_i(u)u}$ to $f_{i,u}$ is bounded by $m|f_i|(2K)^{k(\deg f_i+\deg u)}+(|\alpha_i(u)|-1)m^2|f_i|^2K^{2(\deg f_i+\deg u)}$. Here we first conjugate $u$ to $f_i$ then add $|\alpha_i(u)|$ terms of $f_i^u$. Because $\deg u\leqslant \deg \alpha_i, \sum_{u\in \supp \alpha_i}|\alpha_i(u)|=|\alpha_i|,|\supp \alpha_i|\leqslant |\alpha_i|$. Consequently the cost of the third equality of (\ref{w4}) is bounded by 
\begin{align*}
	&\sum_{u\in \supp \alpha_i} (m|f_i|(2K)^{k(\deg f_i+\deg u)}+(|\alpha_i(u)|-1)m^2|f_i|^2K^{2(\deg f_i+\deg u)})\\
	&\leqslant \sum_{u\in \supp \alpha_i} (m|f_i|(2K)^{k(\deg f_i+\deg \alpha_i)}+(|\alpha_i(u)|-1)m^2|f_i|^2K^{2(\deg f_i+\deg \alpha_i)})\\
	&=|\supp \alpha_i|m|f_i|(2K)^{k(\deg f_i+\deg \alpha_i)}+\sum_{u\in \supp \alpha_i} (|\alpha_i(u)|-1)m^2|f_i|^2K^{2(\deg f_i+\deg \alpha_i)}\\ 
	&= |\supp \alpha_i|m|f_i|(2K)^{k(\deg f_i+\deg \alpha_i)}+(|\alpha_i|-|\supp \alpha_i|)m^2|f_i|^2K^{2(\deg f_i+\deg \alpha_i)}\\
	&\leqslant |\alpha_i|m|f_i|(2K)^{k(\deg f_i+\deg \alpha_i)}+|\alpha_i|m^2|f_i|^2K^{2(\deg f_i+\deg \alpha_i)}\\
	&=|\alpha_i|(m|f_i|(2K)^{k(\deg f_i+\deg \alpha_i)}+m^2|f_i|^2K^{2(\deg f_i+\deg \alpha_i)})\\
	&\leqslant |\alpha_i|(mP(2K)^{kn}+m^2P^2K^{2n}).
\end{align*}
The last inequality is obtained by the condition $\deg(f_i^{\alpha_i})\leqslant n$, i.e $\deg f_i+\deg \alpha_i\leqslant n$.

The forth equality of (\ref{w4}) is adding all $f_{i,u}$'s up. Since 
\[\deg f_{i,u}\leqslant \deg {f_i^{\alpha_i}}\leqslant n,|f_{i,u}|\leqslant |\alpha_i(u)||f_i|\leqslant |\alpha_i||f_i|\leqslant |\alpha_i|P,\]
by \cref{thelemma} (a), the cost of adding $|\supp \alpha_i|$ terms of $f_{i,u}$ is bounded by $(|\supp \alpha_i|-1)m^2 (|\alpha_i|P)^2 K^{2n}$. Here we use the fact that the size of the addition of any step is bounded by $|f_i^{\alpha_i}|$. Therefore the total number of relations we need to convert each $f_i^{\alpha_i}$ to its order form $f_i'$ is bounded by 
\begin{align*}
|\alpha_i|(mP(2K)^{kn}+m^2P^2K^{2n})+(|\supp \alpha_i|-1)m^2 (|\alpha_i|P)^2 K^{2n}\\
\leqslant |\alpha_i|(mP(2K)^{kn}+(1+|\alpha_i|^2)m^2P^2K^{2n}).	
\end{align*}
In general, the cost of converting all $f_i^{\alpha_i}$'s to their order forms is bounded by 
\begin{align*}
	&\sum_{i=1}^l |\alpha_i|(mP(2K)^{kn}+(1+|\alpha_i|^2)m^2P^2K^{2n}) \\
	=&(mP(2K)^{kn}+m^2P^2K^{2n})(\sum_i^l |\alpha_i|)+m^2P^2K^{2n}\sum_{i=1}^l (|\alpha_i|^3)\\
	\leqslant & (mP(2K)^{kn}+m^2P^2K^{2n})(n^2+n)K^{n^{2k}}+m^2P^2K^{2n}(n^2+n)^3K^{3n^{2k}}.
\end{align*}

The next step, as described above, is to add all $f_i'$ up. We have that $|f_i'|\leqslant |\alpha_i||f_i|$ and $\deg f_i'\leqslant n$ for all $i=1,2,\dots, l.$ Moreover, the size of any partial product $\sum_{i=1}^{l'} f_i', 1\leqslant l'\leqslant l$ is controlled by the following inequalities:
\[|\sum_{i=1}^{l'} f_i'|\leqslant \sum_{i=1}^{l'} |\alpha_i||f_i|\leqslant P\sum_{i=1}^{l'} |\alpha_i|\leqslant P(n^2+n)K^{n^{2k}},\deg(\sum_{i=1}^{l'} \alpha_i f_i)\leqslant n.\]
This is similar to add $f_{i,u}$'s. By \cref{thelemma} (a), the cost of the $(|l|-1)$ additions is bounded by 
\[(l-1)m^2(P(n^2+n)K^{n^{2k}})^2K^{2n}\leqslant (l-1)m^2(n^2+n)^2P^2K^{2n^{2k}+2n}.\]
Now we need to verify the process of those steps above indeed result $w_4$. This is provided by the fact $\mu_i=\sum_{j=1}^l \alpha_j\mu_{ji}=\sum_{j=1}^l \mu_{ij}'$ and eventually following \cref{thelemma} we have 
\[\prod_{i=1}^l f_i^{\alpha_i}=\prod_{i=1}^l f_i'=\prod_{i=1}^l (a_1^{\mu_{i1}'}a_2^{\mu_{i2}'}\dots a_m^{\mu_{im}'})=\prod_{j=1}^m a_j^{\sum_{j=1}^l \mu_{ij}'}=a_1^{\mu_1}a_2^{\mu_2}\dots a_m^{\mu_m}=w_4.\]
By our estimation, the cost of the first equality is bounded by $(mP(2K)^{kn}+m^2P^2K^{2n})(n^2+n)K^{n^{2k}}+m^2P^2Q^{2n}(n^2+n)^3K^{3n^{2k}}$ and the cost of the third equality is bounded by $(l-1)m^2(n^2+n)^2P^2K^{2n^{2k}+2n}.$ Other equalities hold in the free group hence no cost. Therefore 
\begin{align*}
	\area(w_4)=&(n^2+n)K^{n^{2k}}+(mP(2K)^{kn}+m^2P^2K^{2n})(n^2+n)K^{n^{2k}}+m^2P^2K^{2n}(n^2+n)^3K^{3n^{2k}}\\
	&+(l-1)m^2(n^2+n)^2P^2K^{2n^{2k}+2n}.
\end{align*}
Now we choose a constant $C>K$ large enough such that 
\begin{align*}
	&(n^2+n)K^{n^{2k}}+(mP(2K)^{kn}+m^2P^2K^{2n})(n^2+n)K^{n^{2k}}+m^2P^2K^{2n}(n^2+n)^3K^{3n^{2k}}\\
	&+(l-1)m^2(n^2+n)^2P^2K^{2n^{2k}+2n}\\
	&\leqslant C^{n^{2k}}
\end{align*}
It is clear that such $C$ exists, for example we can choose $C$ to be $4m^2P^2QK$. Note that $P,Q$ only depends on $f_1,f_2,\dots, f_l$, hence $\mathcal R_4$, and so does $K$. Therefore $C$ is independent of $w$. 

In conclusion, we start with $w=_G 1$ of length at most $n$. By converting it four times, we end up with a word $w_4$, of which area is bounded by $C^{n^{2k}}$. Thus

\[\begin{tikzcd}[row sep=huge, column sep = huge]
    w \arrow[r,"0"] & w_1 \arrow[r,"\leqslant n^2"] & w_2 \arrow[r,"\leqslant (n+n^2)(2K)^n"] & w_3 \arrow[r,"\leqslant (n+n^2)^2K^{2n}"] & w_4.
\end{tikzcd}\]

Summing up all the cost from $w_1$ to $w_4$ and with the fact $C>K$, we conclude that the area of $w$ is bounded above by 
\[\area(w)\leqslant C^{n^{2k}}+(n+n^2)^{2}C^{2n}+(n+n^2)(2C)^{n}+n^2.\] 
This completes the proof.
\end{proof}
\section{Relative Dehn Functions for Metabelian Groups}
\label{relativedehn}

\subsection{The Cost of Metabelianness}
\label{relativedehn1}

First, we state an important consequence of \cref{conjugate}. Consider a finitely presented group $G$ with a short exact sequence
\[1\to A \hookrightarrow G \twoheadrightarrow T \to 1, \]
where $A, T$ are abelian groups. By definition, $G$ is metabelian. The metabelianness is provided by abelianness of $A$. More precisely, if we let $\mathcal A$ be a generating set of $A$ and $\mathcal T$ be a set in $G$ such that their image in $T$ generates $T$, the abelianness of $A$ is given by all commutative relations like $[a^u,b^v]=1$ for all $a,b\in \mathcal A, u,v\in F(\mathcal{T})$. It follows that all commutators are commutative. Recall \cref{conjugate}, then we immediately have 

\begin{theorem}
\label{metabelianness}
	The metabelianness of a finitely presented matabelian group $G$ costs at most exponentially many relations with respect to the length of the word, i.e. there exists a constant $C$ such that 
	\[\area([[x,y],[z,w]])\leqslant C\cdot 2^{|[[x,y],[z,w]]|}, \forall x,y,z,w\in G.\]
\end{theorem}

\begin{proof}
	Let $k$ be the minimal torsion-free rank of an abelian group $T$ such that there exists an abelian normal subgroup $A$ in $G$ satisfying $G/A\cong T$. The projection of $G$ onto $T$ is denoted by $\pi: G\to T$. 

If $k=0$, $G$ has a finitely generated abelian subgroup of finite index. Then the result follows immediately. 

If $k>0$, we first consider the case that $T$ is free abelian. let $\mathcal T=\{t_1,\dots,t_k\}\subset G$ such that $\{\pi(t_1),\dots,\pi(t_k)\}$ forms a basis for $T$ and $\mathcal A$ be a finite subset of $G$ such that it contains all commutators $a_{ij}=[t_i,t_j]$ for $1\leqslant i<j \leqslant k$ and generates the $T$-module $A$. Then $\mathcal A\cup \mathcal T$ is a finite generating set for the group $G$.
Recall that $G$ has a finite presentation as follows
\[G=\langle a_{1},a_2,\dots,a_m, t_1,t_2,\dots,t_k\mid \mathcal R_1\cup \mathcal R_2\cup \mathcal R_3 \cup \mathcal R_4\rangle,\]
where
\begin{align*}
	\mathcal R_1&=\{[t_i,t_j]=a_{ij} \mid 1\leqslant i<j\leqslant k\};\\
	\mathcal R_2&=\{[a,b^u]=1 \mid a,b\in \mathcal A,u\in \bar F, \|\theta(u)\|<R\};\\
	\mathcal R_3&=\{\prod_{u\in \bar F}(a^{\lambda(\hat u)})^u=a\mid a\in \mathcal A,\lambda\in \Lambda\cap C(A)\}\cup \{\prod_{u\in \bar F}(a^{\lambda(\hat u)})^{u^{-1}}=a\mid a\in \mathcal A,\lambda\in \Lambda\cap C(A^*)\};\\
\end{align*}
and $\mathcal R_4$ is the finite set generating $\ker \varphi$. All the notations are the same as in \cref{maintheorem1}. 
 
By \cref{conjugate}, we have that there exists a constant $C_1$ such that 
\[\area{[a,b^u]}\leqslant C_1\cdot 2^{|[a,b^u]|}, a, b\in \mathcal A, u\in F(\mathcal T).\]

Now let $x,y,z,w$ be elements in $G$ and $n=|[[x,y],[z,w]]|$. We use two commutator identities $[a,bc]=[a,c][a,b]^c$ and $[ab,c]=[a,c]^b[b,c]$ to decompose $[x,y]$ and $[z,w]$ into products of $[a,b]^u$ where $a,b\in \mathcal A\cup \mathcal A^{-1}\cup \mathcal T \cup \mathcal T^{-1}, u\in G$. There are three cases to be considered. 
\begin{enumerate}
	\item If $a,b\in \mathcal A \cup \mathcal A^{-1}$, $[a,b]^u=_G 1$ and the cost for converting $[a,b]^u$ to 1 is 1.
	\item If $a,b\in \mathcal T \cup \mathcal T^{-1}$, we have two cases. If $a,b\in \{t_i,t_i^{-1}\}$ for some $i$, $[a,b]^u=_G 1$ with no cost. If $a\in \{t_i,t_i^{-1}\}$ and $b\in \{t_j, t_j^{-1}\}$ where $i\neq j$, $[a,b]^u=c^{\varepsilon u'}$, where $c\in \mathcal A\cup \mathcal A^{-1}, \varepsilon \in \{\pm 1\}, |u'|\leqslant |u|+1$. This is due to cases like $[t_i^{-1},t_j]=[t_i,t_j]^{-t_i^{-1}}$. The cost of converting $[a,b]^u$ to $c^{\varepsilon u'}$ is 1.
	\item If $a\in \mathcal A \cup \mathcal A^{-1}, b\in \mathcal T \cup \mathcal T^{-1}$ (or $b\in \mathcal A \cup \mathcal A^{-1}, a\in \mathcal T \cup \mathcal T^{-1}$), then $[a,b]=_G aa^{b}$ (resp. $[a,b]=b^{a^{-1}}b$ ). Thus $[a,b]^u=_G a^{u}a^{bu}$ (resp. $[a,b]^u=b^{a^{-1}u}b^{u}$). The cost of converting is 0.
\end{enumerate}

It follows that $[x,y] =_G\prod_{i=1}^l b_i^{\varepsilon_i u_i}$, where $l\leqslant 2|x||y|, b_i\in \mathcal A, \varepsilon_i\in \{\pm 1\}, u_i\in G, |u_i|\leqslant |x|+|y|$. The cost of converting $[x,y]$ to $\prod_{i=1}^l b_i^{\varepsilon_i u_i}$ is bounded by $|x||y|$. 

Let $u$ be a word in $G$, we claim that $u=w_1w_2\prod_{i=1}^p c_i^{v_i}$ where $w_1\in F(\mathcal A), w_2\in F(\mathcal T), c_i\in \mathcal A\cup \mathcal A^{-1}, v_i\in F(\mathcal T)$. The claim can be proved by always choosing to commute the left most pair of $ta$ where $t\in \mathcal T\cup \mathcal T^{-1}, a\in \mathcal A\cup \mathcal A^{-1}$. Then for an element $a\in \mathcal A$, we have 
\[a^u=_{F(\mathcal A\cup \mathcal T)} a^{w_1w_2\prod_{i=1}^p c_i^{v_i}}=_G a^{w_2}, w_1\in F(\mathcal A), w_2\in F(\mathcal T).\]
Note that $|w_2|<|u|$. Similar to \cref{organizer}, there exists a constant $C_2$ such that the cost of the second equality in terms of relations is bounded $C_2\cdot 2^{|u|}$.

Therefore $[x,y]=_G\prod_{i=1}^l b_i^{\varepsilon_i u_i'}, $ where $l\leqslant 2|x||y|, b_i\in \mathcal A, \varepsilon_i\in \{\pm 1\}, u_i'\in F(\mathcal T)$. The cost is bounded by $2n^2C_2 2^{n}$. Consequently, $[[x,y],[z,w]]$ can be write as a product of at most $8n^2$ conjugates of elements in $\mathcal A$ at a cost of at most $8n^2C_22^{n}$. Last, converting this product to 1 costs at most $(8n^2)^2C_1 2^n$ by \cref{conjugate}. The theorem is proved in this case.

If $T$ is not free abelian, we suppose that  $\mathcal T=\{t_1,t_2,\dots,t_k, t_{k+1},t_{k+2},\dots,t_s\}$ such that the set $\mathcal T_0:=\{\pi(t_{k+1}),\pi(t_{k+2}),\dots,\pi(t_s)\}$ generates a finite abelian group and $\{\pi(t_{1}),\pi(t_{2}),\dots,\pi(t_k)\}$ generates $\mathbb Z^k$. Then we can write down a presentation of $G$ as following:
 \[G=\langle a_{1},a_2,\dots,a_m, t_1,t_2,\dots,t_s\mid \mathcal R_1\cup \mathcal R_2\cup \mathcal R_3 \cup \mathcal R_4\rangle,\]
where
\begin{align*}
	\mathcal R_1&=\{[t_i,t_j]=a_{ij}, t_l^{n_l}=a_l \mid 1\leqslant i<j\leqslant s,k+1\leqslant l\leqslant s\};\\
	\mathcal R_2&=\{[a,b^u]=1 \mid a,b\in \mathcal A,u\in \bar F, \|\theta(u)\|<R\};\\
	\mathcal R_3&=\{\prod_{u\in \bar F}(a^{\lambda(\hat u)})^u=a\mid a\in \mathcal A,\lambda\in \Lambda\cap C(A)\}\cup \{\prod_{u\in \bar F}(a^{\lambda(\hat u)})^{u^{-1}}=a\mid a\in \mathcal A,\lambda\in \Lambda\cap C(A^*)\};\\
\end{align*}
and $\mathcal R_4$ is the finite set generating $\ker \varphi$. Note that $\theta: \bar F\to \mathbb R^k$ kills all $t_l, l>k$. The rest of the proof is the same as the case when $T$ is free abelian. 	

\end{proof}

\subsection{The Relative Dehn Functions of Metabelian Groups}
\label{relativedehn2}

Recall that a set of groups form a \emph{variety} if it is closed under subgroups, epimorphic images, and unrestricted direct products. The set of metabelian groups naturally form a variety, denoted by $\mathcal S_2$, since metabelian groups satisfy the identity $[[x,y],[z,w]]=1$. Inside a variety, we can talk about relative free groups and relative presentations. Firstly, a metabelian group $M_k$ is \emph{free of rank $k$} if it satisfies the following universal property: every metabelian group generated by $k$ elements is an epimorphic image of $M_k$. It is not hard to show that $M_k\cong F(k)/F(k)''$, where $F(k)$ is a free group of rank $k$ (in the variety of all groups). 

Next, we shall discuss the relative presentations. Recall that the usual presentation of $G$ consists of a free group $F$ and a normal subgroup $N$ such that $G\cong F/N$. For relative presentations, we shall replace the free group by the relative free group. Now let $G$ be a metabelian group generated by $k$ elements, then there exists a epimorphism $\varphi: M_k\to G$, where $M_k$ is generated by $X=\{x_1,x_2,\dots,x_k\}$. We immediately have that $G\cong M_k/\ker \varphi$. Note that $\ker \varphi$ is a normal subgroup of $M_k$, then it is a normal closure of a finite set. We let $R=\{r_1,r_2,\dots,r_m\}$ to be the finite set whose normal closure is the kernel of $\varphi$. Therefore we obtain a \emph{relative presentation} of $G$ in the variety of metabelian groups
\[G=\langle x_1,x_2,\dots,x_k\mid r_1,r_2,\dots,r_m\rangle_{\mathcal S_2}.\]
The notation $\langle \cdot \rangle_{\mathcal S_2}$ is used to indicate that the presentation is relative to the variety of metabelian groups $\mathcal S_2$. Here, the subscript two stands for the derived length two. We denote by $\mathcal P$ the relative presentation $\langle X\mid R\rangle_{\mathcal S_2}$. Note that if $G$ is finitely presented, then the finite presentation in the usual sense is also a relative presentation, with some possible redundant relations. 

Let us give an example of relation presentation of a metabelian group which is not finitely presented. $H_\infty$, the group we introduce in \cref{finitepresented2}, is a free metabelian group of rank $k$. It has two different relative presentations depending on how many generators we choose. 
\[H_\infty=\langle t_1,t_2,\dots, t_k\rangle_{\mathcal S_2}=\langle a_{ij},t_1,t_2,\dots, t_k\mid a_{ij}=[t_i,t_j],1\leqslant i<j\leqslant k\rangle_{\mathcal S_2}.\]

Let $w$ be a word in $G$ such that $w=_G 1$. Then $w$ lies in the normal closure of $R$. Thus $w$ can be written as 
\[w=_{M_k} \prod_{i=1}^l r_i^{f_i} \text{ where }r_i\in R\cup R^{-1},f_i\in M_k.\]
The smallest possible $l$ is called the relative area of $w$, denoted by $\tilde\area_\mathcal P(w)$. The difference between the area and the relative area is that we take the equality in different ambient groups, one in free groups and the other in free metabelian groups. Consequently, the Dehn function relative to the variety of metabelian groups with respect to the presentation $\mathcal P$ is defined as  
\[\tilde\delta_\mathcal P(n)=\sup\{\tilde\area_\mathcal P(w)\mid |w|_{M_k}\leqslant n\}.\]
Here $|\cdot|_{M_k}$ is the word length in $M_k$. Similar to usual Dehn functions, the relative Dehn functions are also independent of finite presentations up to equivalence, i.e.  

\begin{proposition}[\cite{Fuh2000}]
	Let $\mathcal P$ and $\mathcal Q$ be finite relative presentations of the finitely generated metabelian group $G$. Then 
	\[\tilde\delta_\mathcal P \approx \tilde\delta_\mathcal Q.\]
\end{proposition}

Therefore it is valid to denote the relative Dehn function of a finitely generated metabelian group $G$ by $\tilde \delta_G$. One remark is that every finitely generated metabelian group is finite presentable relative to the variety of metabelian groups. Thus the relative Dehn function can be defined for all finitely generated metabelian groups. Another remark is, unlike Dehn functions, it is unknown if the relative Dehn function is a quasi-isometric invariant. The best we can say is the following:

\begin{proposition}
\label{finiteIndexForRelative}
	Let $H,G$ be finitely presented metabelian groups where $H$ is a finite index subgroup of $G$. Then they have the same relative Dehn function up to equivalence.
\end{proposition}

\begin{proof}
	Let $\langle X\mid R\rangle$ be a finite presentation of $H$, where $X=\{x_1,x_2,\dots,x_n\}$. Then we have a finite presentation of $G$ as following:
	\[G=\langle X\cup Y\mid R_0\cup R_1\cup R_2\rangle,\]
	where 
	\begin{align*}
		Y &= \{y_1,y_2,\dots,y_m\};\\
		R_0 &= R; \\
		R_1 &= \{y_iy_j=y_{f(i,j)}w_{i,j}, y_i^{-1}=y_{g(i)}u_i \}, w_{i,j}, u_i\in (X\cup X^{-1})^*\\
		f: &\{1,2,\dots,m\}\times \{1,2,\dots,m\}\to \{1,2,\dots,m\}, g:\{1,2,\dots,m\}\to \{1,2,\dots,m\};\\
		R_2 &= \{x_ly_i=y_jv_{i,j}\}, v_{l,i}\in (X\cup X^{-1})^*.
	\end{align*}
	We claim that there exists a constant $L$ such that for every word $w=_G 1$, there exists a word $w'$ such that $w'=w, w'\in (X\cup X^{-1})^*$ and $|w'|\leqslant L|w|$. Moreover, it costs at most $|w|$ relations from $R_1\cup R_2$ to convert $w$ to $w'$.
	
	If the claim is true, then we have that 
	\[\tilde \area_G (w)\leqslant \tilde \area_H(w')+|w|.\]
	It immediately implies that 
	\[\tilde\delta_G(n)\leqslant \tilde\delta_H(Ln)+n.\]
	Thus 
	\[\tilde\delta_G(n)\preccurlyeq \tilde\delta_H(n).\]
	And the other direction $\tilde\delta_H(n)\preccurlyeq \tilde\delta_G(n)$ is obvious since $w_H=1$ implies $w_G =1$. 
	
	To prove the claim, we let $L=\max \{|w_{i,j}|,|u_i|,|v_{l,i}|\mid 1\leqslant i,j\leqslant m, 1\leqslant l\leqslant n\}.$ Let $w$ be a word such that $w=_G 1$. WLOG, we assume that $w$ has the following form:
	\[w=_{F(X\cup Y)} a_1b_1a_2b_2\dots a_kb_ka_{k+1}, a_i\in (X\cup X^{-1})^*, b_i\in (Y\cup Y^{-1})^*,\]
	where only $a_1,a_{k+1}$ might be empty word. For $b_k$, using relations in $R_1$ we have that 
	\[b_k=y_{h(k)}b_k', h(k)\in \{1,2,\dots,m\},b_k\in (X\cup X^{-1})^*,\]
	and $|b_k'|\leqslant L|b_k|$. Thus,
	\[w=a_1b_1a_2b_2\dots a_ky_{h(k)}b_k'a_{k+1},\]
	while the cost of converting is bounded by $|b_k|$ and all relations are from $R_1$. 
	
	Next we commute $y_{h(k)}$ with $a_k$ using relations from $R_2$.
	\[a_ky_{h(k)}=y_{h(k)}a_k', a_k'\in (X\cup X^{-1})^*,\]
	and $|a_k'|\leqslant L|a_k|$. Substituting it in, we get
	\[w=a_1b_1a_2b_2\dots y_{h(k)}a_k'b_k'a_{k+1},\]
	while the cost of converting is bounded by $|a_k|$ and all relations are from $R_2$. 
	
	Therefore, repeating the above process, we eventually have
	\[w=y_{h(1)}a_1'b_1'a_2'b_2'\dots a_k'b_k'a_{k+1}.\]
	Since $w=1$, thus $y_{h(1)}$ is actually an empty word. Consequently, we have
	\[w=a_1'b_1'a_2'b_2'\dots a_k'b_k'a_{k+1}\in (X\cup X^{-1})^*,\]
	and the length of the left-hand side is controlled by 
	\[|a_1'b_1'a_2'b_2'\dots a_k'b_k'a_{k+1}|\leqslant \sum_{i=1}^k L(|a_i|+|b_i|)+|a_{k+1}|\leqslant L|w|.\]
	The cost of relations is bounded by $\sum_{i=1}^{k} (|a_i|+|b_i|)\leqslant |w|$ while all relations are from $R_1\cup R_2$. The claim is proved.

\end{proof}

Let us consider one classic example: the Baumslag-Solitar group $BS(1,2)$. The relative presentation is the same as the usual presentation $BS(1,2)=\langle a,t\mid a^t=a^2\rangle_{\mathcal S_2}$. But one can prove that the relative Dehn function of $BS(1,2)$ is $n$ instead of the usual Dehn function $2^n$ \cite{Fuh2000}. In general, it is difficult to compute the relative Dehn function of a finitely generated metabelian group. We will list some known examples in \cref{relativedehn4}. 

So what is the connection between the relative Dehn function and Dehn function? To answer this question, we have to go back to the complexity of the membership problem of the submodule that we discussed in \cref{maintheorem4}. 

\subsection{Connections Between Dehn Functions and Relative Dehn Functions}
\label{relativedehn3}

The goal of the section is to prove the following theorem:
\begin{theorem}
\label{main3}
Let $G$ be a finitely presented metabelian group. Then 
	\[\tilde \delta_G(n)\preccurlyeq \delta_G(n)\preccurlyeq \max\{\tilde \delta_G^3(n^3),2^n\}.\]
\end{theorem}

Let $k$ be the minimal torsion-free rank of an abelian group $T$ such that there exists an abelian normal subgroup $A$ in $G$ satisfying $G/A\cong T$.

First, if $k>0$ we notice that the problem can be reduced in the same way as \cref{main2} does. Because for a finitely presented metabelian group $G$ there exists a subgroup $G_0$ of finite index such that $G_0$ is an extension of an abelian group by a free abelian group of rank $k$. Most importantly, by \cref{finiteindex} and \cref{finiteIndexForRelative}, their Dehn functions are equivalent as well as their relative Dehn functions. Therefore from now on, we assume that $G$ is an extension of an abelian group $A$ by a free abelian group $T$. The projection of $G$ onto $T$ is denoted by $\pi: G\to T$. 

Let $\mathcal T=\{t_1,\dots,t_k\}\subset G$ such that $\{\pi(t_1),\dots,\pi(t_k)\}$ forms a basis for $T$ and $\mathcal A$ be a finite subset of $G$ such that it contains all commutators $a_{ij}=[t_i,t_j]$ for $1\leqslant i<j \leqslant k$ and generates the $T$-module $A$. Then $\mathcal A\cup \mathcal T$ is a finite generating set for the group $G$.
Recall that $G$ has a finite presentation as follows
\[G=\langle a_{1},a_2,\dots,a_m, t_1,t_2,\dots,t_k\mid \mathcal R_1\cup \mathcal R_2\cup \mathcal R_3 \cup \mathcal R_4\rangle,\]
where
\begin{align*}
	\mathcal R_1&=\{[t_i,t_j]=a_{ij} \mid 1\leqslant i<j\leqslant k\};\\
	\mathcal R_2&=\{[a,b^u]=1 \mid a,b\in \mathcal A,u\in \bar F, \|\theta(u)\|<R\};\\
	\mathcal R_3&=\{\prod_{u\in \bar F}(a^{\lambda(\hat u)})^u=a\mid a\in \mathcal A,\lambda\in \Lambda\cap C(A)\}\cup \{\prod_{u\in \bar F}(a^{\lambda(\hat u)})^{u^{-1}}=a\mid a\in \mathcal A,\lambda\in \Lambda\cap C(A^*)\};\\
\end{align*}
and $\mathcal R_4$ is the finite set generating $\ker \varphi$. All the notations are the same as in \cref{maintheorem1}. 

Since we are dealing with relative Dehn function, we can reduce amount of redundant relations in $\mathcal R_2$. We set $\mathcal R'_2=\{[a,b]=1,[a,b^t]=1\mid a,b\in \mathcal A, t\in \mathcal T\}$. Then we have

\begin{lemma}
\label{relativeCommutative}
	$\mathcal R_2'$ generates all commutative relations $[a,b^u]=1, a,b\in \mathcal A, u\in F(\mathcal T)$ in the presentation relative to the variety of metabelian groups. Moreover, the relative area of $[a,b^u]$ is bounded by $4|u|-3$.
\end{lemma}

\begin{proof}
	Suppose the result is proved for $|u|\leqslant n$, i.e., $[a,b^u]=1$ can be written as a product of conjugates of words in $\mathcal R'_2$ and metabelian relations. For metabelian relations, we mean those relations make commutators commute to each other. Note that the relative area of any metabelian relations is 0.
	
	Now for that case $|u|=n+1$, let $u=vt, |v|=n, t\in \{t_1,t_2,\dots,t_k\}$. By metabelian relations, we have that 
	\[1=[a^{-1}a^t,b^{-t}b^u].\]
	Since $a^{-1}a^t=[a,t]$ and $b^{-t}b^u=[b,v]^t$.
	Then by inductive assumption, we are able to use relations like $[a,b^w]=1$ when $a,b\in \mathcal A, |w|\leqslant n$. In particular, $[a,b^v]=1$.
	
	And notice that 
	\begin{align*}
		1=[a^{-1}a^t,b^{-t}b^u] &= \underbrace{a^{-t}a}_{\text{commute}}b^{-u}b^{t}\underbrace{a^{-1}a^t}_{\text{commute}}b^{-t}b^u \\
		&= a\underbrace{a^{-t}b^{-u}}_{\text{commute}}\underbrace{b^{t}a^t}_{\text{commute}}a^{-1}b^{-t}b^u \\
		&= ab^{-u}a^{-t}a^t\underbrace{b^{t}a^{-1}}_{\text{commute}}b^{-t}b^u \\
		&= ab^{-u}a^{-1}b^u.
	\end{align*}
	This shows that $[a,b^u]$ can be generated by $\mathcal R_2'$ and matebelian relations. Let us count the cost. In the computation above we use $[a,b^{v}]=1$ once (notice that $[a^t,b^u]=[a,b^v]^t$), $[a,a^t]=1$ twice, $[a,b]=1$ once, and $[a,b^t]$ once. Therefore, 
	\[\tilde\area([a,b^u])\leqslant \tilde\area([a,b^v])+4\leqslant 4(|v|+1)-3=4(n+1)-3.\]
	This completes the proof.
\end{proof}

The lemma allows us to replace $\mathcal R_2$ by $\mathcal R_2'$ in the relative presentation. And we immediately get the relative version of \cref{conjugate}.

\begin{lemma}
\label{relativeConjugate}
	Let $u$ be a reduced word in $F(\mathcal T)$ and $\bar u$ be the unique word in $T$ representing $u$ in the form of $t_1^{m_1}t_2^{m_2}\dots t_k^{m_k}$. Then we have
	\[\tilde\area(a^{-u}a^{\bar u})\leqslant 4|u|^2+2|u|.\]
\end{lemma}

\begin{proof}
	The only difference of this proof to the proof of \cref{conjugate} is that now it only costs $4|u|-3$ to commute conjugates every time. 
\end{proof}

Thus in the relative sense, we save a lot of cost due to the fact we assume metabelianness is free of charge. 

Next, we focus on the $T$-module $A$. It is not hard to see that $A$ is the quotient of the free $T$-module generated by $\mathcal A$ by the submodule generated by $\mathcal R_3\cup \mathcal R_4$. We then replace $\mathcal R_3\cup \mathcal R_4$ by the Gr\"{o}bner basis $\mathcal R'_3=\{f_1,f_2,\dots,f_l\}$ for the same submodule. Therefore we finally have the relative presentation of $G$ we want:
\[G=\langle a_1,a_2,\dots,a_m, t_1,t_2,\dots,t_k\mid \mathcal R_1\cup \mathcal R_2'\cup \mathcal R_3'\rangle_{\mathcal S_2}.\]

We let $M$ be the free $T$-module generated by $\mathcal A$ and $S$ be the submodule generated by $\mathcal R_3'$ over $T$. So that $A\cong M/S$. For an element $f=\mu_1 a_1+\mu_2 a_2+\dots +\mu_m a_m$ in $M$, we define its length, denoted by $\|f\|$, to be the length of the word $a_1^{\mu_1}a_2^{\mu_2}\dots a_m^{\mu_m}$ in $G$. Then for every element $f$ in $S$, there exists $\alpha_1,\alpha_2,\dots,\alpha_l\in R$ such that 
\[f=\alpha_{1}f_1+\alpha_2 f_2+\dots+\alpha_l f_l.\]
We denote by $\area_A(f)$ the minimal possible $\sum_{i=1}^l |\alpha_i|$. Then \emph{the Dehn function} of the submodule $A$ is defined to be 
\[\delta_A(n)=\max\{\area_A(f)\mid \|f\|\leqslant n\}.\]

Then we have a connection between the relative Dehn function of $G$ and the Dehn function of the submodule $S$.

\begin{lemma}
\label{relativeConnection}
	Let $G$ be a finitely presented metabelian group and $A$ is defined as above, then 
	\[\delta_A(n)\preccurlyeq \tilde \delta_G(n) \preccurlyeq \max\{\delta_A^3(n^3), n^6\}. \]
\end{lemma}

\begin{proof}
	Now let $w$ be a word of length $n$ such that $w=_G 1$. We then estimate the cost of converting it to the ordered form. The process is exactly the same as in the proof of \cref{main2}. We replace the cost by the cost in relative presentation by \cref{relativeCommutative} and \cref{relativeConjugate}. It is not hard to compute that it costs at most $n^2+(4n-3)(n^2+n)+(4n-3)^2(n^2+n)^2$ to convert $w$ to its ordered form $w':=\prod_{i=1}^m a_i^{\mu_i}$ where $\sum_{i=1}^n |\mu_i|\leqslant n^2,\deg \mu_i\leqslant n$ and $|w'|\leqslant 2n^3$. Since $w'$ lies in the normal subgroup generated by $\mathcal R_3'$, then there exists $\alpha_1,\alpha_2,\dots,\alpha_l$ such that 
	\[w'=\prod_{i=1}^l f_i^{\alpha_i}, \sum_{i=1}^l |\alpha_i|\leqslant \delta_A(2n^3).\]
	The relative area of the left hand side is less than $\sum_{i=1}^l |\alpha_i|$. Then we just repeat the same process of the proof of \cref{main2}, and compute the cost of adding $f_i^{\alpha_i}$ up to $w'$. The cost is bounded by $(\sum_{i=1}^l |\alpha_i|)^3$ up to equivalence. Thus the relative area of $w'$ is asymptotically bounded by $\delta_A^3(n^3)$ up to equivalence. And hence the relative area of $w$ is bounded by $\max\{\delta_A^3(n^3),n^6\}$. Thus the right inequality is proved. 
	
	For the left inequality in the statement, let $\prod_{i=1}^m a_i^{\mu_i}$ be a word of ordered form such that it realizes $\delta_A(n)$. The length of the word, by definition, is bounded by $n$. We claim that the relative area of $\prod_{i=1}^m a_i^{\mu_i}$ is greater than $\delta_A(n)$. If not, by the definition of the relative area, we have that 
	\[\prod_{i=1}^m a_i^{\mu_i}=\prod_{j=i}^{s} r_i^{h_i}, r_i\in \mathcal R_1'^{\pm 1}\cup \mathcal R_2'^{\pm} \cup \mathcal R_3'^{\pm}, h_i\in M_{m+k},\]
	where $s=\area(\prod_{i=1}^m a_i^{\mu_i})<\delta_S(n)$. If we only keep all relations from $\mathcal R_3'$ and combine the same relations together, we will get $\prod_{i=1}^l f_i^{\alpha_i}$ and $\sum_{i=1}^l |\alpha_i|\leqslant s<\delta_A(n)$. Since canceling relations like $[t_i,t_j]=a_{ij}, [a,b^t]=1$ and commuting $f_i^{h_j}$'s do not change the value of left hand side as an element in free $T$-module generated by basis $\{a_1,a_2,\dots,a_m\}$. Therefore we eventually get 
	
	\[\sum_{i=1}^m \mu_i a_i=\sum_{j=1}^l \alpha_j f_j, \sum_{i=1}^l |\alpha_i|<\delta_A(n).\]
	
	It leads to a contradiction. 
\end{proof}

Since the proof of \cref{relativeConnection} works even if the group is not finitely presented, by replacing $G$ by a finitely generated metabelian group, we have

\begin{proof}[Proof of \cref{relativeDehnfuntion}]
Let $G$ be a finitely generated metabelian group. If $k=0$, $G$ has a finitely generated abelian subgroup of finite index. Then the relative Dehn function is asymptotically bounded by $n^2$.

If $k>0$, similarly, we can reduce the case to that $G$ is an extension of a module $A$ by a free abelian group $T$ such that the torsion-free rank of $T$ is $k$. Then a word $w=_G 1$ with $|w|\leqslant n$ can be converted to its ordered form $w':=\prod_{i=1}^m a_i^{\mu_i}$ where $|w|\leqslant 2n^3, \deg(w)\leqslant n, \sum_{i=1}^m |\mu_i|\leqslant n^2$. Then by \cref{division}, there exists a word $w''$ such that $w'=_G w'', \area(w'')\leqslant 2^{n^{2k}}$. The theorem follows immediately.
\end{proof}

Finally we have all the ingredients to prove \cref{main3}.

\begin{proof}[Proof of \cref{main3}]
	The left inequality is obvious since the finite presentation of $G$ is also the relative finite presentation of $G$. 
	
	If $k=0$, $G$ has a finitely generated abelian subgroup of finite index. The result follows immediately.
	
	If $k>0$, let $w$ be a word of length $n$ and $w=_G 1$. Then there exists $\alpha_1,\alpha_2,\dots,\alpha_l$ 
	\begin{equation}
	\label{relative}
		w=\prod_{i=1}^l f_i^{\alpha_i},\sum_{i=1}^l |\alpha_i|\leqslant \delta_A(2n^3),\deg \alpha_i\leqslant n.
	\end{equation}
	According to the proof of \cref{main2}, adding the left hand side of (\ref{relative}) costs at most $\max\{\delta_A^3(2n^3),2^n\}$ up to equivalence. All other steps of converting cost at most exponential with respect to $n$. Then by the left inequality in \cref{relativeConnection}, $\area(w)\leqslant \max\{\tilde\delta_G^3(n^3),2^n\}$. Therefore the theorem is proved.  
\end{proof}

\subsection{Computing Relative Dehn Functions}
\label{relativedehn4}

First, let us list some known results for relative Dehn functions.

\begin{theorem}[\cite{Fuh2000}]
\label{knownResult}
	\begin{enumerate}[(1)]
		\item The wreath product of two finitely generated abelian groups has polynomial relative Dehn function.
		\item The Baumslag-Solitar group $BS(1,2)$ has linear Dehn function.
		\item Let $G=B(n,m)=\langle a,t \mid (a^n)^t = a^m \rangle_{\mathcal S_2}$ where $m>2, m=n+1$. Then $\tilde\delta_G(n)\preccurlyeq n^3$.
	\end{enumerate}
\end{theorem}

There is one result in (\cite{Fuh2000}, Theorem E) we can improve. 

\begin{proposition}
\label{improved}
	Let $T$ be a finitely generated abelian group and let $A$ be a finitely generated $T$-module. Form the semidirect product
	\[G=A\rtimes T.\]
	Then $\delta_G(n)\preccurlyeq \max\{n^3, \delta_A^3(n^2)\}.$	
\end{proposition}
	
\begin{proof}
	It is not hard to reduce the problem to the case when $T$ is free abelian. Thus we just assume that $T$ is a finitely generated free abelian group. Suppose $\mathcal T=\{t_1,t_2,\dots,t_k\}$ is a basis of $T$ and $\mathcal A=\{a_1,a_2,\dots,a_m\}$ generates the module $A$ over $\mathbb ZT$. Let $M$ be the free $T$-module generated by $\mathcal A$ and $S$ be a submodule of $M$ generated by $f_1,f_2,\dots,f_l$, where $f_i=\sum_{j=1}^m \alpha_{i,i} a_{j}$ for $1\leqslant i\leqslant k, \alpha_{i,j}\in \mathbb ZT$. Then we can write down a presentation of $G$ as follows
	\begin{align*} 
		G=\langle &a_1,a_2,\dots,a_m,t_1,t_2,\dots,t_m \mid [t_i,t_j]=1 (1\leqslant i<j \leqslant k), \\
		&[a_i,a_j^{w}]=1 (1\leqslant i<j \leqslant m, w\in \mathbb ZT), \prod_{j=1}^m a_j^{\alpha_{i,j}}=1 (1\leqslant i \leqslant l)\rangle.
	\end{align*}
	Then, by the same discussion as in \cref{relativedehn3}, we have a finite relative presentation of $G$:
	\begin{align*} 
		G=\langle &a_1,a_2,\dots,a_m,t_1,t_2,\dots,t_m \mid [t_i,t_j]=1 (1\leqslant i<j \leqslant k), \\
		&[a_i,a_j]=1, [a_i,a_j^{t_s}]=1 (1\leqslant i<j \leqslant m, 1\leqslant s \leqslant k), \prod_{j=1}^m a_j^{\alpha_{i,j}}=1 (1\leqslant i \leqslant l)\rangle_{\mathcal S_2}.
	\end{align*}
	Now let $w=_G 1$ and $|w|\leqslant n$. Since in this case all $t_i,t_j$ commutes, it is much easier than the general case. Following the same process as in the proof of \cref{main2}, $w$ can be convert to its ordered form $a_1^{\mu_1}a_2^{\mu_2}\dots a_m^{\mu_m}$, where $\deg(\mu_i)<n, \sum_{i=1}^m |\mu_i|\leqslant n$. The cost is bounded by $n^3$. Notice that the length of $a_1^{\mu_1}a_2^{\mu_2}\dots a_m^{\mu_m}$ is bounded by $n^2$. Then there exists $\alpha_1,\alpha_2,\dots,\alpha_l\in \mathbb ZT$ such that 
	\[a_1^{\mu_1}a_2^{\mu_2}\dots a_m^{\mu_m}=\prod_{i=1}^l f_i^{\alpha_i}, \sum_{i=1}^l |\alpha_i|\leqslant \delta_A(n^2).\]
	The rest of the proof is the same as the proof of \cref{relativeConnection}, since it is just a special case of \cref{relativeConnection}.
\end{proof}

\section{Examples and Further Comments}
\label{examples}

\subsection{Subgroups of Metabelian Groups with Exponential Dehn Function}
\label{examples1} 

It is time to implement our technique to some concrete examples and investigate all the obstacles preventing us to construct a finitely presented metabelian group with Dehn function that exceeds exponential function.

The class of examples we investigate in this section was introduced by Baumslag in 1973 \cite{Baumslag1973}. Let $A$ be a free abelian group of finite rank freely generated by $\{a_1,a_2,\dots,a_r\}$. Furthermore let $T$ be a finitely generated abelian group with basis $\{t_1,t_2,\dots,t_k,\dots,t_l\}$, where $t_1,\dots,t_k$ are of infinite order and $t_{k+1},\dots,t_l$ are respectively of finite order $m_{k+1},\dots,m_l$. Finally let $F=\{f_1,f_2,\dots,f_k\}$ be a set of element $f_i$ from $\mathbb ZT$, where each $f_i$ is of the form
\[f_i=1+c_{i,1}t_i+c_{i,2}t_i^2+\dots+c_{i,d_i-1}t_i^{d_{i}-1}+t_i^{d_i}, d_i\geqslant 1, c_{i,j}\in \mathbb Z.\]
Now let us define a group $W_F$ corresponds to $F$. The generating set is the following
\[X=\{a_1,a_2,\dots,a_r,t_1,t_2,\dots,t_l,u_1,\dots,u_k\},\]
where $r,k,l$ are the same integers as above. 

The defining relations of $W_F$ are of four kinds. First we have the power relations
\[t_i^{m_i}=1, i=k+1,\dots,l.\]
Next we have the commutativity relations 
\[\begin{cases}
[u_i,u_j]=1, &1\leqslant i,j\leqslant k;\\
[t_i,t_j]=1, &1\leqslant i,j\leqslant l;\\
[t_i,u_j]=1, &1\leqslant i\leqslant l, 1\leqslant j\leqslant k;\\
[a_i,a_j]=1, &1\leqslant i,j\leqslant r.
\end{cases}
\]
Thirdly we have the commutativity relations for the conjugates of the generators $a_i$:
\[[a^u_i,a^w_j]=1, 1\leqslant i,j\leqslant r,\]
where $u,w\in \{t_1^{\alpha_1}t_2^{\alpha_2}\dots t_l^{\alpha_l}\mid 0\leqslant \alpha_i\leqslant d_i\text{ for }i=1,\dots,k, 0\leqslant \alpha_i<m_i\text{ for }i=k+1,\dots,l\}.$
Finally we have relations defining the action of $u_j$ on $a_i$:
\[a_i^{u_j}=a_i^{f_j}, 1\leqslant i\leqslant r, 1\leqslant j\leqslant k. \]
It is not hard to show that $W_F$ is metabelian \cite{Baumslag1973}. Moreover, Baumslag showed the following:

\begin{proposition}[\cite{Baumslag1973}]
Given a free abelian group $A$ of finite rank and a finite generated abelian group $T$, there exists $F$ such that $A\wr T \hookrightarrow W_F$.
\end{proposition}

In particular, if $r=k=l$ and we let $f_i=1+t_i$ for all $i$, $W_F$ contains a copy of the free metabelian group of rank $r$.

We claim that
\begin{proposition}
\label{exponential}
$W_F$ has an exponential Dehn function.
\end{proposition}

Note that when $i=j=k=1$, $f_1=1+t_1$, $W_F$ is the Baumslag group $\Gamma=\langle a,s,t \mid [a,a^t]=1,[s,t]=1, a^s=aa^t \rangle$. The exponential Dehn function of this special case is proved in \cite{KassabovRiley2010}.

We need a few lemmas before we prove \cref{exponential}. First, let us denote the abelian groups generated by $\{t_1,t_2,\dots.t_l\}$ and $\{u_1,u_2,\dots,u_k\}$ by $T$ and $U$ respectively. 

\begin{lemma}
\label{submoduble}
	Let $M$ be a free $(U\times T)$-module with basis $e_1,\dots,e_r$. Let $S$ be the submodule of $M$ generated by $\{(u_i-f_i)e_j\mid 1\leqslant i\leqslant k, 1\leqslant j \leqslant r\}$. If $h=h_1e_1+h_2e_2+\dots +h_r e_r \in S$ such that $h_i\in \mathbb ZT$ for all $i$. Then $h=0.$
\end{lemma}

\begin{proof}
	If $k=1$, then $h\in S$ means there exists $\alpha_1,\alpha_2,\dots,\alpha_r\in \mathbb Z(U\times T)$ such that 
	\[h=\alpha_1(u_1-f_1)e_1+\alpha_2(u_1-f_1)e_2+\dots+\alpha_r(u_1-f_1)e_r.\]
	Since $h=h_1e_1+h_2e_2+\dots +h_r e_r$, then $h_i=\alpha_i(u_1-f_1)$. Note that $h_i\in \mathbb ZT$. It follows that $\alpha_i(u_1-f_1)$ does not have any term involves $u_1$. Suppose $\alpha_i\neq 0$ for some $i$. Because $f_1\in \mathbb ZT$, $\deg_{u_1}(\alpha_i u_1)>\deg_{u_1}(\alpha_i f_1)$. Thus $\alpha_i(u_1-f_1)$ has at least one term contains $u_1$, that leads to a contradiction. 
	
	If the statement of $k=n$ has been proved, for $k=n+1$, we have
	\[h=\sum_{i=1}^r\sum_{j=1}^{n+1} \alpha_{i,j}(u_j-f_j)e_i.\]
	We choose an integer $N$ large enough such that $u_1^N\alpha_{i,j}$ does not have any negative power of $u_1$ for all $i,j$. Then
	\[u_1^Nh=\sum_{i=1}^r\sum_{j=1}^{n+1} f_1^N\alpha_{i,j}(u_j-f_j)e_i=:\sum_{i=1}^r\sum_{j=1}^{n+1} \beta_{i,j}(u_j-f_j)e_i,\]
	where $\beta_{i,j}=u_1^{N}\alpha_{i,j}$. We regard $\beta_{i,j}(u_1)$ as a polynomial of $u_1$. Replacing $u_1$ by $f_1$, we have
	\[f_1^N h=\sum_{i=1}^r \sum_{j=2}^{n+1} \beta_{i,j}(f_i)(u_j-f_j)e_i.\]
	Note that $f_1^N h_i\in \mathbb ZT$ for $i=1,\dots, r$, then by the inductive assumption, $f_1^N h_i=0$ for all $i$. Since $f_1=1+c_{1,1}t_1+c_{1,2}t_1^2+\dots+c_{i,d_1-1}t_1^{d_{1}-1}+t_1^{d_1}$ and $t_1$ has infinite order, then $f_1$ is not a zero divisor in $\mathbb Z(U\times T)$. Thus $h_i=0$ for all $i$.
	
	Therefore $h=0$. The induction finishes the proof.
\end{proof}

It follows that if $a_1^{h_1}a_2^{h_2}\dots a_r^{h_r}=_{W_F} 1$ such that $h_i\in \mathbb ZT$ for all $i$, then $h_i=0$ as an element in $\mathbb Z(U\times T)$ for every $i$. To convert it to 1, we only need those metabelian relations to commute all the conjugates of $a_i$'s. By \cref{metabelianness}, it will cost at most exponentially many relations with respect to the length of the word to kill the word.

Next, let $w=_{W_F}1$ and consider the minimal van Kampen diagram $\Delta$ over $W_F$. There are two types of relations contain $u_i$: (1) commutative relations $[u_i,u_j]=1, [u_i,t_s]=1, j\neq i, 1\leqslant s\leqslant l$; (2) action relations $a_j^{u_i}=a_j^{f_i}, 1\leqslant j\leqslant r$. Those cells, in the van Kampen diagram, form a $u_i$-band. 

\begin{figure}[H]
		\centering
			\includegraphics[width=15cm]{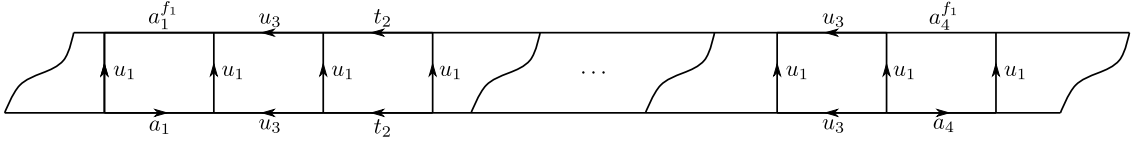}
			\caption{an example of a $u_1$-bands}
\end{figure}

We have some properties for $u_i$-bands in a van Kampen diagram over $W_F$.

\begin{lemma}
\label{bands}
\begin{enumerate}[(i)]
	\item The top (or bottom) path of a $u_i$-band is a word $w$ that all $t_s, u_j$ for $s,j\neq i$ are in the same orientation, i.e. the exponents of each letter $t_s,u_j$'s are either all 1 or all $-1$. In particular, 
	\[w=_{W_F}a_1^{h_1}a_2^{h_2}\dots a_r^{h_r}t_1^{\alpha_1}\dots t_l^{\alpha_l}u_1^{\beta_1}\dots u_k^{\beta_k},\]
	where $h_i\in \mathbb Z(U\times T)$, $\sgn(\alpha_i)=\sgn(\beta_j)$ for all $i,j$, and $\alpha_s$ (or $\beta_j$) is equal to the number of times of $t_s$ (resp. $u_j$) appears in $w$ for $s,j\neq i$.
	\item $u_i$-bands do not intersect each other. In particular, a $u_i$-band does not self-intersect. 
	\item If $i\neq j$, a $u_i$-band intersects a $u_j$-band at most one time. 
\end{enumerate}	
\end{lemma}

\begin{proof}
	\begin{enumerate}[(i)]
		\item By the definition of a $u_i$, all letters $t_s,u_j, s,j\neq i$ of the top (or bottom) path must share the same direction. The second half of the statement can be proved basically the same way as we did for the ordered form (See \cref{maintheorem2}). 		
		\item Because there is no $u_i$ on the top or the bottom path of a $u_i$-band, two $u_i$-bands cannot intersect each other. 
		\item If $i\neq j$ and a $u_i$-band intersects a $u_j$-band. Since the van Kampen diagram is a planer graph, by comparing the orientation, it is impossible for a $u_i$-band to intersect a $u_j$-band twice (or more).
	\end{enumerate}
\end{proof}

Last, we have

\begin{lemma}
\label{tensor}
	Let $f(t)=t^d+c_{d-1}t^{d-1}+\dots+c_1t+1\in \mathbb Z[t], c_i\in \mathbb Z, d>0$. Then there exists $\alpha>1$ such that $|(f(t))^n|>\alpha^n$ for all $n$.
\end{lemma}

\begin{proof}
	We denote that $(f(t))^n=\sum_{i=0}^{nd} c_{n,i}t^i$. 
	
	Consider the corresponding holomorphic function $g(z)=z^d+c_{1,d-1}z^{1,d-1}+\dots+c_{1,1}z+1$. If $\exists z_0, |z_0|=1$ such that $|g(z_0)|>1$, we have
	\[|g(z_0)|^n=|(g(z_0))^n|=|\sum_{i=0}^{nd} c_{n,i} z_0^i|\leqslant \sum_{i=0}^{nd} |c_{n,_i}|=|f^n|.\]
	Then we are done. 
	
	Now suppose $|g(z)|\leqslant 1$ for all $|z|=1$. Then by Cauthy's integral formula we have
	\[1=g(0)=\frac{1}{2\pi\mathrm{i}} \int_{|z|=1} \frac{g(z)}{z}dz.\]
	Take modulus on both sides:
	\[1=|\frac{1}{2\pi\mathrm{i}} \int_{|z|=1} \frac{g(z)}{z}dz|\leqslant \frac{1}{2\pi}\int_{\theta=0}^{2\pi} |g(e^{\mathrm{i}\theta})|d\theta\leqslant 1.\]
	Therefore $|g(z)|=1$ for $|z|=1$ almost everywhere. Let $z=e^{\mathrm{i\theta}}$, we have
	\[g(e^{\mathrm{i}\theta})=(\sum_{j=0}^d c_{1,j}\cos(i\theta))+\mathrm{i}(\sum_{j=0}^d c_{1,j}\sin(j\theta)).\]
	Then
	\[(\sum_{j=0}^d c_{1,j}\cos(i\theta))^2+(\sum_{j=0}^d c_{1,j}\sin(j\theta))^2=\sum_{h=0}^d c_{1,j}^2+2\sum_{j<k} c_{1,j}c_{1,k}\cos((k-j)\theta)=1, \]
	holds for all $\theta$. But $\cos((k-j)\theta)$ is a polynomial with respect to $\cos \theta$, i.e. $\cos((k-j)\theta)=T_{k-j}(\cos\theta)$, where $T_{m}(x)$ is the $m$-th Chebyshev polynomial. The leading term of $T_m(x)$ is $2^{m-1}x^m$. Thus
	\[\sum_{h=0}^d c_{1,j}^2+2\sum_{j<k} c_{1,j}c_{1,k}T_{k-j}(\cos \theta)=1, \forall \theta.\]
		Note the leading term of the left-hand side is $2^{d-1}\cos^d \theta$. That leads a contradiction since the equation above has at most $d$ solution for $\cos \theta$.
\end{proof}

\begin{proof}[Proof of \cref{exponential}]

First, we show that the lower bound is exponential. Consider the word $w=[a_1^{u_1^n},a_1]$. $w$ is of length $2n+4$ and $w=_{W_F}1$. Let $\Delta$ be a minimal Van-Kampen diagram with boundary label $w$. By comparing the orientation, $u_1$-bands starting at the top left of $\Delta$ will end at either bottom left or top right. By \cref{bands}, $u_1$-bands do not intersect each other, then we can suppose at least half of the $u_1$-bands starting at the top left end at the top right. See in Figure 4, the shaded areas are $u_1$-bands.

\begin{figure}[H]
		\centering
			\includegraphics[width=16cm]{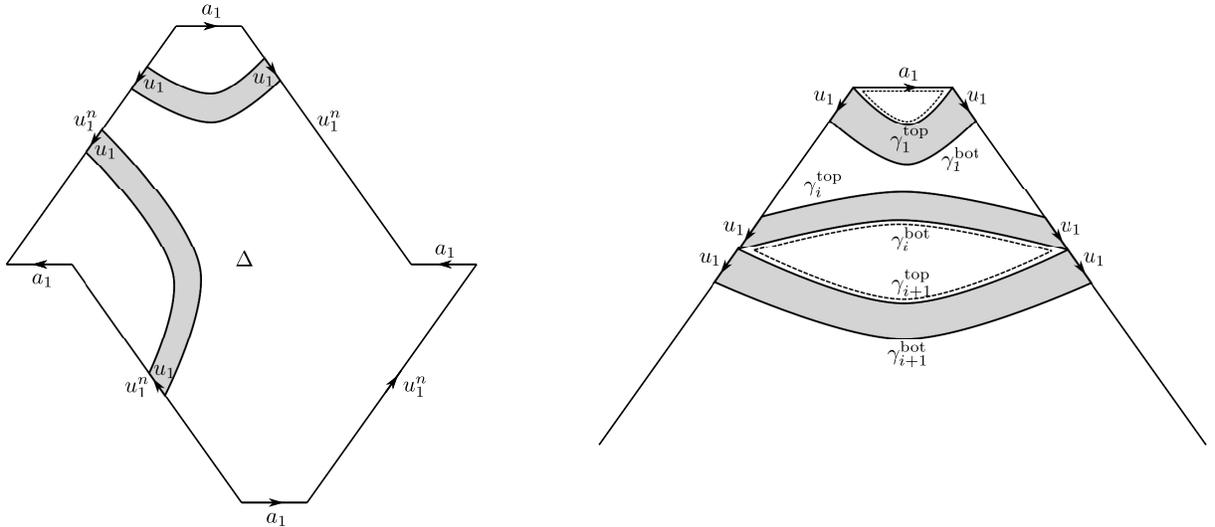}
			\caption{$u_1$-bands in $\Delta$}
\end{figure}

We first claim that there are no cells containing $t_s,u_j, s,j>1$ on each $u_1$-band. We denote the top and bottom path of the $i$-th $u_1$-band from the top by $\gamma_i^{\mathrm{top}}$ and $\gamma_i^{\mathrm{bot}}$, where $i=1,2,\dots,m$, $m>\frac{n}{2}$. Assuming a $u_i$-band intersect one of the $u_1$-band, again by \cref{bands} (ii), (iii), it can neither intersect a $u_1$-band twice nor intersect itself. Thus it has to end all the way to the boundary of $\Delta$. A contradiction. 

Then if there exists a cell containing $t_s$ for $s>1$ in the top most $u_1$-band, then by \cref{bands} (i), $\gamma_1^{\mathrm{top}}$ is a word $a_1^{h_1}a_2^{h_2}\dots a_r^{h_r}t_1^{\alpha_1}\dots t_l^{\alpha_l}$. Thus $\gamma_1^{\mathrm{top}}$ and $a_1$ form a cycle $\gamma$. We have 
    \[a_1^{h_1+1}a_2^{h_2}\dots a_r^{h_r}t_1^{\alpha_1}\dots t_l^{\alpha_l}=1, \alpha_i\neq 0.\]
    It leads to a contradiction since the image of the left hand side in $U\times T$ is not trivial. And by definition of a $u_1$-band, if $\gamma_{i}^{\mathrm{top}}$ does not have any $t_s, s>1$, neither does $\gamma_{i}^{\mathrm{bot}}$. Next consider two consecutive $u_1$-bands. If $\gamma_i^{\mathrm{bot}}$ does not have $t_s$, then by the same argument, neither does $\gamma_{i+1}^{\mathrm{top}}$. Therefore the claim is true. 

	Denote the words of $\gamma_i^{\mathrm{top}}$ and $\gamma_i^{\mathrm{bot}}$ by $w_i^{\mathrm{top}}$ and $w_i^{\mathrm{bot}}$ respectively. Such words only consist of $a_i$'s and $t_1$. Note that $w_i^{\mathrm{bot}}=w_{i+1}^{\mathrm{top}}$ for $i=1,\dots, m-1$. Since $w_1^{\mathrm{bot}}=a_1^{-f_1}$, by the same discussion above, $w_i^{\mathrm{top}}=a_1^{-f_1^{i-1}}, w_i^{\mathrm{bot}}=a_1^{-f_1^i}$ (See in Figure 4). Next we focus on the number of $a_1$ in each $w_i^{\mathrm{bot}}$, which is at least $|f_1^i|$. By \cref{tensor}, there exists $\alpha>1$ such that $|f_1^i|>\alpha^i$. Therefore, the number of $a_1$ in $w_m^{\mathrm{top}}$ is at least $\alpha^{m-1}$. Since $m>\frac{n}{2}$, the number of cells in the $m$-th $u_1$-band is at least $\alpha^{[\frac{n}{2}]}$. Thus the area of $[a_1^{u_1^{n}},a_1]$ is at least $\alpha^{[\frac{n}{2}]}$. It follows that the lower bound is exponential. 

For the upper bound, as \cref{metabelianness} suggests, all we need is to consider how to solve the membership problem of the submodule $S$ where $S$ is generated by $\{(u_i-f_i)e_j\mid 1\leqslant i\leqslant k, 1\leqslant j \leqslant r\}$. Suppose $w=1$ with $|w|\leqslant n$, the $w$ has a ordered form as 
	\[w=_{W_F}a_1^{g_1}a_2^{g_2}\dots a_r^{g_r}, g_i\in \mathbb Z(U\times T).\]
	And the cost of converting $w$ to its ordered form is exponential with respect to $n$ as we showed in \cref{maintheorem}. Also note that $\deg(g_i),|g_i|\leqslant n$ for all $i$. WLOG, we assume that the all exponents of $u_i$'s are positive. The corresponding module element of $w$ is 
	\[g_1e_1+g_2e_2+\dots+g_re_r.\]
	For each term $t_1^{\alpha_1}t_2^{\alpha_2}\dots t_l^{\alpha_l}u_1^{\beta_1}u_2^{\beta_2}\dots u_k^{\beta_k}$, $\alpha_i\in \mathbb Z,\beta_i\geqslant 0$, we replace $u_i$ by $u_i-f_i+f_i$. Then we convert $t_1^{\alpha_1}t_2^{\alpha_2}\dots t_l^{\alpha_l}u_1^{\beta_1}u_2^{\beta_2}\dots u_k^{\beta_k}$ to a form 
	\[\sum_{i=1}^k \eta_i(u_i-f_i)+\tau, \eta_i\in \mathbb Z(U\times T), \tau\in \mathbb ZT.\]
	If $|\alpha_1|+\dots+|\alpha_l|+|\beta_1|+\dots+|\beta_k|<n$, then $\deg(\eta_i),\deg(\tau)<Dn, |\eta_i|,|\tau|<D^n$, where $D=\max\{d_1,\dots,d_k, |f_1|,\dots,|f_k|\}$. Therefore, replacing $u_i$ by $u_i-f_i+f_i$ in every term of $w$, we have
	\[g_1e_1+g_2e_2+\dots+g_re_r=\sum_{i=1}^r\sum_{i=1}^k \mu_{i,j}(u_j-f_j)e_i+\rho, \mu_{i,j}\in \mathbb Z(U\times T),\rho\in \mathbb ZT.\]
	Since $w=1$, then $\rho$ lies in the submodule $S$. By \cref{submoduble}, $\rho=0$. Also note that $\deg(\mu_{i,j}),\deg(\rho)<Dn,|\mu_{i,j}|,|\rho|<nD^n$. It follows from \cref{thelemma} that all module computations in the process cost exponentially many relations with respect to $n$. And it also cost at exponentially many relations to convert $\rho$ to 0. Therefore
	\[w=_{W_F} \prod_{i=1}^r \prod_{j=1}^k a_i^{\mu_{i,j}(u_j-f_j)},\]
	and the cost of converting is exponential with respect to $n$. And the area of the right hand side is bounded by $\sum_{i,j}|\mu_{i,j}|\leqslant rknD^n$. The upper bound is exponential. 
\end{proof}

\cref{embeddingWreathproduct} follows immediately from \cref{exponential}.

\subsection{Further Comments}
\label{examples2}

\cref{metabelianness} shows that the metabelianness costs at most exponential. By \cref{main3} and \cref{relativeConnection}, if we write $G$ as an extension of two abelian groups $A$ by $T$, the complexity of the membership problem of the $T$-module $A$ gives the lower bound of $\delta(n)$. It follows that to construct a finitely presented metabelian group with Dehn function bigger than exponential function, the only hope is to find a complicated membership problem of a submodule in a free module over a group ring of a free abelian group. Because it is simply impossible to get anything harder than exponential anywhere else.

The first obstacle for us is the existence of such a membership problem. There is already a lot of study of the polynomial ideal membership problem, which is the special case for the membership problem over modules. For example, Mayr and Meyer showed that the lower space bound of a general polynomial ideal membership problem is exponential \cite{mayr1982complexity}. Other results can be found several surveys, such as \cite{mayr1997some}, \cite{mayr2017complexity}. But it remains unknown whether there exists an integral coefficient polynomial ideal membership problem for which the time complexity is harder than exponential. 

The second obstacle comes from the finitely-presentedness. Recall that a finitely generated metabelian group is finitely presented if and only if the module structure is tame. Thus even if we manage to find a complicated enough membership problem for submodules, it may not give us a finitely presented metabelian group unless the module is tame. 

\medskip

\bibliography{bib}{}
\bibliographystyle{alpha}

\end{document}